\documentclass[12pt]{amsart}

\newtheorem{theorem}[equation]{Theorem}
\newtheorem{lemma}[equation]{Lemma}
\newtheorem{corollary}[equation]{Corollary}
\newtheorem{proposition}[equation]{Proposition}

\numberwithin{equation}{section}

\usepackage{amscd,amsmath,amssymb,verbatim}

\begin{document}

\title[codimension-one $A$-hypergeometric systems]{On solutions of codimension-one $A$-hypergeometric systems}
\author{Alan Adolphson}
\address{Department of Mathematics\\
Oklahoma State University\\
Stillwater, Oklahoma 74078}
\email{alan.adolphson@okstate.edu}
\author{Steven Sperber}
\address{School of Mathematics\\
University of Minnesota\\
Minneapolis, Minnesota 55455}
\email{sperber@math.umn.edu}
\date{\today}
\begin{abstract}
By a codimension-one system we mean a system whose lattice of relations has rank one.  We consider codimension-one $A$-hypergeometric systems and  explicitly construct some of the logarithmic series solutions at the origin.  When the parameter vector $\beta$ is nonresonant, we obtain a full set of logarithmic series solutions at the origin by this procedure.  We also determine when a codimension-one system with nonresonant parameter can have maximal unipotent monodromy at the origin.
\end{abstract}
\maketitle

\section{Introduction}

Let $A=\{{\bf a}_1,\dots,{\bf a}_n\}\subseteq{\mathbb Z}^d$ and let $V\subseteq{\mathbb R}^d$ be the subspace spanned by $A$.  We assume that $\dim  V = n-1$ and that each subset of $A$ of cardinality $n-1$ is linearly independent.  The lattice of relations on $A$ then has rank one and is generated by a single relation
\begin{equation}
\sum_{i=1}^k \ell_i{\bf a}_i - \sum_{j=k+1}^n \ell_j{\bf a}_j = {\bf 0},
\end{equation}
where the coefficients $\ell_i$ and $\ell_j$ are relatively prime positive integers.  
Put $x_0 = \prod_{i=1}^k x_i^{\ell_i}/\prod_{j=k+1}^n x_j^{\ell_j}$.  
Our focus will be on finding \emph{logarithmic series solutions\/} of the $A$-hypergeometric system at 
$x_0=0$.  (The definition of the $A$-hyper\-geometric system is recalled in Section~5.)  By a logarithmic series solution at $x_0=0$, we mean a solution in a Nilsson ring ${\mathbb C}[[x_0]][x_0^{-1},\log x_0,x^v]$, where $v\in{\mathbb C}^n$ and $x^v = \prod_{\mu=1}^n x_\mu^{v_\mu}$.  

The associated toric ideal $I_A$ is the principal ideal in the Weyl algebra generated by $\prod_{i=1}^k \partial_i^{\ell_i}-\prod_{j=k+1}^n \partial_j^{\ell_j}$.  To find solutions at $x_0=0$, we can take any weight vector $w$ for which the corresponding initial ideal ${\rm in}_w(I_A)$ is the principal ideal generated by $\prod_{i=1}^k \partial_i^{\ell_i}$.  For example, we can take $w=(1,0,\dots,0)$.  (For any unexplained teminology or notation, we refer the reader to Saito, Sturmfels, and Takayama \cite[Chapter~3]{SST}.)  

It is easy to calculate the fake exponents.  Let $\beta\in V_{\mathbb C}$, where $V_{\mathbb C}\subseteq{\mathbb C}^d$ is the ${\mathbb C}$-span of the set $A$.  Let $i\in\{1,\dots,k\}$, $b\in\{0,1,\dots,\ell_i-1\}$, and define $v^{(i,b)}=(v^{(i,b)}_1,\dots,v^{(i,b)}_n)\in{\mathbb C}^n$ to be the unique vector satisfying the conditions $v^{(i,b)}_i=b$ and $\sum_{\mu=1}^n v^{(i,b)}_\mu{\bf a}_\mu=\beta$.  (There is a unique such vector because the set $A\setminus\{{\bf a}_i\}$ is a basis of $V_{\mathbb C}$.)  The set
\[ E_\beta = \{ v^{(i,b)}\mid i\in\{1,\dots,k\},\ b\in\{0,1,\dots,\ell_i-1\}\}, \]
of cardinality $\leq \sum_{i=1}^k \ell_i$, is the set of fake exponents of the $A$-hyper\-geometric system with parameter $\beta$.  If $\beta$ is sufficiently generic, e.~g., if $v^{(i,b)}_\mu\not\in{\mathbb Z}$ for $\mu\neq i$, then the elements of $E_\beta$ are exponents and $\lvert E_\beta\rvert = \sum_{i=1}^k \ell_i$.  For sufficiently generic $\beta$ there is associated to each $v^{(i,b)}\in E_\beta$ a canonical series solution that lies in the Nilsson ring ${\mathbb C}[[x_0]][x^{v^{(i,b)}}]$ (see \cite[Section~3.4]{SST}), giving a total of $\sum_{i=1}^k \ell_i$ canonical series solutions. 

Let $\Delta(A)\subseteq V$ be the convex hull of $A\cup\{{\bf 0}\}$ and let ${\mathbb Z}A\subseteq V$ be the abelian group generated by $A$.  Let ${\rm vol}(A)$ be the volume of $\Delta(A)$ relative to Lebesgue measure on $V$ normalized so that a simplex defined as the convex hull of the origin and a basis for ${\mathbb Z}A$ has volume~$1$.    We show in Section 2 that
\begin{equation}
{\rm vol}(A) = \max\bigg\{\sum_{i=1}^k \ell_i,\sum_{j=k+1}^n \ell_j\bigg\}.
\end{equation}
When $\beta$ is sufficiently generic and
\begin{equation}
{\rm vol}(A) =\sum_{i=1}^k \ell_i,
\end{equation}
we thus get a total of ${\rm vol}(A)$ independent canonical series solutions in the Nilsson rings ${\mathbb C}[[x_0]][x^{v^{(i,b)}}]$, so we think of $x_0=0$ as a regular singularity when (1.3) holds and as an irregular singularity when (1.3) fails.

{\bf Example.} (The classical one-variable hypergeometric series.)  Take $A = \{{\bf a}_1,\dots,{\bf a}_{d+1}\}$ with 
\[ {\bf a}_1 = (1,\dots,1,-1,\dots,-1) \in{\mathbb Z}^d, \]
where $1$ is repeated $d+1-k$ times and $-1$ is repeated $k-1$ times and let ${\bf a}_2,\dots,{\bf a}_{d+1}$ be the standard unit basis vectors in ${\mathbb R}^d$ in the reverse of the usual ordering:
\[ {\bf a}_2 = (0,\dots,0,1),\dots,{\bf a}_{d+1} = (1,0,\dots,0). \]
Then
\[ {\bf a}_1+{\bf a}_2 +\cdots+{\bf a}_k - ({\bf a}_{k+1} + \cdots + {\bf a}_{d+1}) = {\bf 0} \]
and all $\ell_i$ and $\ell_j$ equal 1.  In this case the $A$-hypergeometric system is the system satisfied (for generic parameters) by
\[ {}_{d+1-k}F_{k-1}(\alpha_1,\dots,\alpha_{d+1-k};\gamma_1,\dots,\gamma_{k-1};x) = \sum_{s=0}^\infty \frac{(\alpha_1)_s\cdots(\alpha_{d+1-k})_s}{(\gamma_1)_s\cdots(\gamma_{k-1})_s s!}x^s \]
(see the last two items of Dwork-Loeser \cite[Appendix]{DL}: one takes the parameter 
\[ \beta = (-\alpha_1,\dots,-\alpha_{d+1-k},\gamma_1-1,\dots,\gamma_{k-1}-1).) \]
We have $\sum_{i=1}^k \ell_i=k$, so the condition that $k\geq d+1-k$ implies by (1.2) that ${\rm vol}(A) = k$.  The requirement that $k\geq d+1-k$ is the condition for $x=0$ to be a regular singularity of this hypergeometric equation.

We return to the general setting.  For arbitrary $\beta\in V_{\mathbb C}$, one expects the dimension of the space of logarithmic series solutions at $x_0=0$ to be $\sum_{i=1}^k \ell_i$ since that is the dimension in the generic case.  (There are no exceptional parameters in the codimension-one case.)  Our goal is to give explicit formulas for these solutions when $\beta$ is nonresonant.

Put $E'_\beta = \{ v=(v_1,\dots,v_n)\in E_\beta\mid \text{$v_i\not\in{\mathbb Z}_{<0}$ for $i=1,\dots,k$}\}$.
For $v\in E'_\beta$, put $M_v = \{i\in\{1,\dots,k\}\mid v_i\in{\mathbb Z}_{\geq 0}\}$, a nonempty set by the definition of $E_\beta$.  We define the {\it multiplicity\/} $m_v$ of $v$ by $m_v = \lvert M_v\rvert$.  We prove in Section 3 that $E'_\beta$ is nonempty and that
\begin{equation}
\sum_{v\in E'_\beta} m_v = \sum_{i=1}^k \ell_i.
\end{equation}

Every $v\in E'_\beta$ gives rise to a canonical series solution in the ring ${\mathbb C}[[x_0]][x^v]$ of the $A$-hypergeometric system with parameter $\beta$.  When $v\in E'_\beta$ satisfies certain conditions (see Theorem~6.5) we construct $m_v$ logarithmic series solutions of the $A$-hypergeometric system with parameter $\beta+u$ (where $u\in{\mathbb Z}A$) that are polynomials in $\log x_0$ of degrees $0,1,\dots,m_v-1$ with coefficients in ${\mathbb C}[[x_0]][x_0^{-1},x^v]$.  If all $v\in E'_\beta$ satisfy these conditions, then the set of all such solutions has cardinality $\sum_{i=1}^k \ell_i$ by (1.4).  We postpone stating these conditions until Section~6, but all nonresonant $\beta$ will satisfy them, so for nonresonant $\beta$ we get a full set of logarithmic series solutions at the origin (see Section~7).  It will be clear from our explicit formulas that the logarithm-free parts of these solutions have positive radius of convergence.

Other authors have described more general methods for finding logarithmic series solutions of $A$-hypergeometric systems: see Saito, Sturmfels, and Takayama \cite{SST}, Dickenstein, Martinez, and Matusevich \cite{DMM}, or Saito \cite{S}.  Although our method applies to a more restricted class of equations, we feel it is worth presenting because of its simplicity in comparison with other approaches and because of the explicit formulas obtained.

This paper is organized as follows.  In Section 2 we record some simple properties of the polytope $\Delta(A)$ which are needed in the proofs.  In particular, we prove Equation~(1.2).  In Section 3, we collect some properties of $E'_\beta$ and give a proof of Equation~(1.4).  In Sections 4 and~5 we give a construction of formal logarithmic solutions for general $A$-hypergeometric systems.  This extends the results of \cite{AS1}, where we generalized the notion of minimal negative support to construct formal logarithmic solutions of orders 1 and 2, and also the results of \cite[Section~7]{AS}, where we needed only a special case of Theorem~5.22.  In Section 6, we examine these formal solutions under the conditions of Section 1 and prove our main result, Theorem~6.5, which determines which of those formal solutions belong to a Nilsson ring.  In Section 7 we specialize to the case where $\beta$ is nonresonant.  We show that $\beta$ then satisfies the hypothesis of Corollary~6.6, so we get a complete set of logarithmic series solutions.  In Section 8 we prove a result relating $E'_\beta$ and $E'_\gamma$ when $\beta$ is nonresonant and $\gamma\in \beta + {\mathbb Z}A$.  

Our original motivation when beginning this project was to determine for which sets $A$ of lattice points the $A$-hypergeometric system can have maximal unipotent monodromy.  This question is treated in Section 9.  

\section{The polytope $\Delta(A)$}

In this section we collect some results about $\Delta(A)$ to verify (1.2) and to use later in the construction of solutions.  
Since the argument is symmetric, it suffices to prove that if
\begin{equation}
\sum_{i=1}^k \ell_i\geq \sum_{j=k+1}^n \ell_j,
\end{equation}
then (1.3) holds.  So from now until the end of the proof of Lemma~2.10, we assume that (2.1) holds.

For $i\in\{1,\dots,k\}$, let $\Delta_i$ be the convex hull of the origin and the set~$A\setminus\{{\bf a}_i\}$.  Each $\Delta_i$ is an $(n-1)$-simplex with a vertex at the origin.
\begin{lemma}
One has $\Delta(A) = \bigcup_{i=1}^k \Delta_i$.
\end{lemma}

\begin{proof}
Let $x\in\Delta(A)$.  We need to show that $x\in\Delta_i$ for some $i\in\{1,\dots,k\}$.  Since $x\in\Delta(A)$, there exist $c_1,\dots,c_n\in{\mathbb R}_{\geq 0}$ with $\sum_{\mu=1}^n c_\mu\leq 1$ such that
\begin{equation}
x=\sum_{\mu=1}^n c_\mu{\bf a}_\mu.
\end{equation}
Choose $i\in\{1,\dots,k\}$ such that the ratio $c_i/\ell_i$ is minimal.  To fix ideas, suppose that $c_{1}/\ell_{1}\leq c_i/\ell_i$ for $i=2,\dots,k$.  Solve (1.1) for ${\bf a}_{1}$ and substitute into (2.3):
\begin{equation}
 x = \sum_{i=2}^k \bigg(c_i - c_{1}\frac{\ell_i}{\ell_{1}}\bigg){\bf a}_i + \sum_{j=k+1}^n \bigg(c_j+c_{1}\frac{\ell_j}{\ell_{1}}\bigg){\bf a}_j. 
\end{equation}
The minimality of $c_{1}/\ell_{1}$ implies that all the coefficients in (2.4) are nonnegative.  The sum of these coefficients is
\[ {c_{1}} \bigg(\sum_{j=k+1}^n \frac{\ell_j}{\ell_{1}} - \sum_{i=2}^{k} \frac{\ell_i}{\ell_{1}}\bigg) +\sum_{\mu=2}^n c_\mu. \]
Inequality (2.1) implies that the coefficient of $c_{1}$ in this expression is less than or equal to $1$, so this expression is less than or equal to $1$.  Equation (2.4) then shows that $x\in\Delta_{1}$.  
\end{proof}

We extend the definition of the $\Delta_i$.  For $I\subseteq\{1,\dots,k\}$, let $\Delta_I$ be the convex hull of the origin and the set $A\setminus\{{\bf a}_i\}_{i\in I}$.  Thus $\Delta_I$ is an $(n-|I|)$-simplex.  
\begin{lemma}
For $I\subseteq\{1,\dots,k\}$ one has 
\begin{equation}
\bigcap_{i\in I} \Delta_i = \Delta_I.
\end{equation}
\end{lemma}

\begin{proof}
Clearly $\Delta_I\subseteq \bigcap_{i\in I}\Delta_i$.  We prove the reverse inclusion by induction on $\lvert I\rvert$.  For $\lvert I\rvert=1$ the result is trivial.  Suppose the reverse inclusion holds for sets of cardinality $\mu$ and let $\lvert I\rvert = \mu+1$.  To fix ideas, suppose that $I=\{1,\dots,{\mu+1}\}$.  Let $I' = I\setminus\{{\mu+1}\}$.  We have by induction
\begin{equation}
\bigcap_{i\in I} \Delta_i = \bigg(\bigcap_{i\in I'} \Delta_i\bigg)\cap \Delta_{\mu+1} = \Delta_{I'}\cap\Delta_{\mu+1}.
\end{equation}
Let $x\in \Delta_{I'}\cap\Delta_{\mu+1}$.  Then we can write
\begin{equation}
 x = \sum_{\nu=\mu+1}^n c_\nu{\bf a}_\nu = \sum_{\substack{\nu=1\\ \nu\neq \mu+1}}^n d_\nu{\bf a}_\nu
\end{equation}
where
\[ 0\leq c_\nu,d_\nu\leq 1 \text{ and } \sum_\nu c_\nu\leq 1 \text{ and } \sum_\nu d_\nu\leq 1. \]
We get by subtraction
\[ {\bf 0} = \sum_{\nu=1}^\mu(-d_\nu){\bf a}_\nu + c_{\mu+1}{\bf a}_{\mu+1} +\sum_{\nu=\mu+2}^n (c_\nu-d_\nu){\bf a}_\nu. \]
By (1.1), the $\{-d_\nu\}_{\nu=1}^\mu$ and $c_{\mu+1}$ must all have the same sign.  If they are all nonnegative, then $d_\nu=0$ for $\nu=1,\dots,\mu$, so the second equation of (2.8) expresses $x$ as an element of $\Delta_I$.  If they are all nonpositive, then $c_{\mu+1}=0$, so the first equation of (2.8) expresses $x$ as an element of $\Delta_I$.  In either case we get $\bigcap_{i\in I}\Delta_i\subseteq \Delta_I$.
\end{proof}

It follows from Lemmas 2.2 and 2.5 that
\begin{equation}
{\rm vol}(A) = \sum_{i=1}^k {\rm vol}(\Delta_i).
\end{equation}
The following lemma then completes the proof of (1.3).
\begin{lemma}
For $i=1,\dots,k$ one has ${\rm vol}(\Delta_i)=\ell_i$.
\end{lemma}

\begin{proof}
Let $\tilde{\Delta}_i$ be the $(n-1)$-parallelotope associated to $\Delta_i$:
\[ \tilde{\Delta}_i = \bigg\{\sum_{\substack{\mu=1\\ \mu\neq i}}^n t_\mu{\bf a}_\mu\mid \text{$0\leq t_\mu<1$ for all $\mu$}\bigg\}. \]
Then $(n-1)!\text{vol}(\Delta_i) = \text{vol}(\tilde{\Delta}_i)$.  Furthermore,
\begin{equation}
\frac{\text{vol}(\tilde{\Delta}_i)}{(n-1)!} = [{\mathbb Z}A:\langle {\bf a}_1,\dots,\hat{\bf a}_i,\dots,{\bf a}_n\rangle],
\end{equation}
where $\langle {\bf a}_1,\dots,\hat{\bf a}_i,\dots,{\bf a}_n\rangle$ denotes the subgroup of ${\mathbb Z}A$ generated by the set $A\setminus\{{\bf a}_i\}$.  It follows from (1.1) that the index of this subgroup is $\ell_i$: each coset is represented by one element of the set 
\[ \{{\bf 0}, {\bf a}_i,2{\bf a}_i,\dots, (\ell_i-1){\bf a}_i\}. \] 
\end{proof}

At this point we drop the assumption that (2.1) holds.

We also need a description of the facets (the codimension-one faces) of $\Delta(A)$ that contain the origin.  The facets of $\Delta(A)$ containing the origin will also contain $n-2$ elements of $A$.  For $\mu,\nu\in\{1,\dots,n\}$, $\mu\neq \nu$, let $H_{\mu\nu}$ be the hyperplane of $V$ containing the origin and the points $A\setminus\{{\bf a}_\mu,{\bf a}_\nu\}$.  Associated to $H_{\mu\nu}$ is a linear function $h_{\mu\nu}$ on $V$ (unique up to sign) defined by the conditions 
\[ h_{\mu\nu}({\mathbb Z}A) = {\mathbb Z} \]
and $h_{\mu\nu}$ vanishes on $H_{\mu\nu}$:
\[ h_{\mu\nu}({\bf a}_\sigma) = 0\quad\text{for $\sigma\neq \mu,\nu$.} \]
This function is extended to $V_{\mathbb C}$ by defining for $c_1,\dots,c_n\in{\mathbb C}$
\[ h_{\mu\nu}\bigg(\sum_{\sigma=1}^n c_\sigma{\bf a}_\sigma\bigg) = \sum_{\sigma=1}^n c_\sigma h_{\mu\nu}({\bf a}_\sigma) = c_\mu h_{\mu\nu}({\bf a}_\mu) + c_\nu h_{\mu\nu}({\bf a}_\nu). \]

\begin{lemma}
The facets of $\Delta(A)$ that contain the origin lie in the hyperplanes $H_{ij}$ for $i\in\{1,\dots,k\}$ and $j\in\{k+1,\dots,n\}$.
\end{lemma}

\begin{proof}
Applying $h_{ij}$ to Equation (1.1) gives $\ell_ih_{ij}({\bf a}_i) = \ell_jh_{ij}({\bf a}_j)$.  This implies that $h_{ij}({\bf a}_i)$ and $h_{ij}({\bf a}_j)$ are both positive or both negative, hence ${\bf a}_i$ and ${\bf a}_j$ lie on the same side of $H_{ij}$.  Thus $H_{ij}$ contains the facet of $\Delta(A)$ containing $A\setminus\{{\bf a}_i,{\bf a}_j\}$.  

A similar argument shows that for $i,i'\in\{1,\dots,k\}$ the points ${\bf a}_i$ and ${\bf a}_{i'}$ lie on opposite sides of $H_{ii'}$ and for $j,j'\in\{k+1,\dots,n\}$ the points ${\bf a}_j$ and ${\bf a}_{j'}$ lie on opposite sides of $H_{jj'}$, so these hyperplanes do not contain facets.
\end{proof}

Recall that $\beta\in V_{\mathbb C}$ is {\it nonresonant\/} if $\beta + {\mathbb Z}A$ contains no points on the hyperplanes $H_{ij}$ determined by the facets of $\Delta(A)$ containing the origin.  Equivalently, by Lemma~2.12, this says that $h_{ij}(\beta)\not\in{\mathbb Z}$ for all $1\leq i\leq k$ and all $k+1\leq j\leq n$.  This condition will be applied via the following lemma.
\begin{lemma}
Suppose that $\beta\in V_{\mathbb C}$ is nonresonant and let $v\in{\mathbb C}^n$ satisfy $\sum_{\mu=1}^n v_\mu{\bf a}_\mu=\beta$.  If $v_i\in{\mathbb Z}$ for some $i\in\{1\dots,k\}$, then $v_j\not\in{\mathbb Z}$ for $j=k+1,\dots,n$.
\end{lemma}

\begin{proof}
Let $i\in \{1,\dots,k\}$ be such that $v_i\in{\mathbb Z}$.  Suppose that $v_j\in{\mathbb Z}$ for some $j\in\{k+1,\dots,n\}$.  Then
\[ h_{ij}(\beta) = h_{ij}\bigg(\sum_{\mu=1}^n v_\mu{\bf a}_\mu\bigg) = v_ih_{ij}({\bf a}_i) + v_jh_{ij}({\bf a}_j)\in{\mathbb Z}, \]
contradicting the assumption that $\beta$ is nonresonant.
\end{proof}

\section{The set $E'_\beta$}

Put $\ell = (\ell_1,\dots,\ell_k,-\ell_{k+1},\dots,-\ell_n)$.
The following lemma implies in particular that $E'_\beta\neq\emptyset$.
\begin{lemma}
For each $v\in E_\beta$ there exists a unique $z_0\in{\mathbb Z}$ such that $v+z_0\ell\in E'_\beta$.
\end{lemma}

\begin{proof}
Let $v\in E_\beta$.  Then $v=v^{(i,b)}$ for some pair $(i,b)$ with $0\leq b<\ell_i$.  Let $z_0$ be the least integer such that 
for $i'=1,\dots,k$ the $i'$-th coordinate of $v^{(i,b)} + z_0\ell$ is not a negative integer.  For $z>z_0$ the vector $v^{(i,b)} + z\ell$ is not in $E_\beta$ because for $i'=1,\dots,k$ the $i'$-th coordinate cannot lie in the set $\{0,1,\dots,\ell_{i'}-1\}$.  

Fix an $i'\in\{1,\dots,k\}$ for which the $i'$-th coordinate of $v^{(i,b)}+(z_0-1)\ell$ is a negative integer.  Then the $i'$-th coordinate of $v^{(i,b)}+z\ell$ is a negative integer for all $z<z_0$.  This proves that $v^{(i,b)} + z\ell\not\in E'_\beta$ for all $z<z_0$.

For that fixed $i'$, let $b'$ be the $i'$-th coordinate of $v^{(i,b)}+z_0\ell$.  We must have $0\leq b'< \ell_{i'}$, so $v^{(i,b)} + z_0\ell = v^{(i',b')}$, which proves that $v^{(i,b)} + z_0\ell\in E'_\beta$.  
\end{proof}

For $v\in E_\beta$, we set $\tilde{v} = v+z_0\ell\in E'_\beta$, where $z_0$ is the unique integer for which $v+z_0\ell\in E'_\beta$.  

\begin{lemma}
Every element of $E'_\beta$ has minimal negative support, hence is an exponent.
\end{lemma}  

\begin{proof}
Let $v^{(i,b)}\in E'_\beta$.  Its negative support is contained in the set $\{k+1,\dots,n\}$.  For $z>0$ the negative support of $v^{(i,b)}+z\ell$ cannot be a proper subset of the negative support of $v^{(i,b)}$ because the $j$-th coordinate of $\ell$ is a negative integer for $j=k+1,\dots,n$.  
For $z<0$, the negative support of $v^{(i,b)}+z\ell$ is not a proper subset of the negative support of $v^{(i,b)}$ because its $i$-th coordinate is a negative integer.  
\end{proof}

\begin{proof}[Proof of Equation $(1.4)$]
Let 
\[ T = \{(i,b)\mid \text{$i\in\{1,\dots,k\}$, $b\in\{0,1,\dots,\ell_i-1\}$}\}, \] 
a set of cardinality $\sum_{i=1}^k \ell_i$.  From the definition of $E_\beta$, we have a surjective map $T\to E_\beta$ defined by $(i,b)\mapsto v^{(i,b)}$. From Lemma~3.1 we have a surjective map $E_\beta\to E'_\beta$, which we denoted by $v^{(i,b)}\mapsto \tilde{v}^{(i,b)}$.  
The composition of these two maps gives a surjective map $T\to E'_\beta$, where $(i,b)\mapsto\tilde{v}^{(i,b)}$.  To prove (1.4), we need to show that for $(i,b), (i',b')\in T$ we have $\tilde{v}^{(i,b)} = \tilde{v}^{(i',b')}$ if and only if both the $i$-th and $i'$-th coordinates of $\tilde{v}^{(i,b)}$ are nonnegative integers.

One direction of this implication is clear.  By construction, the $i$-th coordinate of $\tilde{v}^{(i,b)}$ lies in ${\mathbb Z}_{\geq 0}$ and the $i'$-th coordinate of $\tilde{v}^{(i',b')}$ lies in~${\mathbb Z}_{\geq 0}$.  So the equality $\tilde{v}^{(i,b)} = \tilde{v}^{(i',b')}$ implies that both the $i$-th and $i'$-th coordinates lie in ${\mathbb Z}_{\geq 0}$.  

Conversely, let $\tilde{v}^{(i,b)}\in E'_\beta$ and suppose that the $i'$-th coordinate of $\tilde{v}^{(i,b)}$ is also a nonnegative integer.  There exists a unique integer $z$ such that the $i'$-th coordinate of $\tilde{v}^{(i,b)}+z\ell$ equals some $b'\in\{0,1,\dots,\ell_{i'}-1\}$.  This implies that $\tilde{v}^{(i,b)}+z\ell=v^{(i',b')}$ by the definition of $v^{(i',b')}$.  Then $v^{(i',b')}-z\ell = \tilde{v}^{(i,b)}\in E'_\beta$, so by Lemma~3.1 $-z$ is the unique integer such that $v^{(i',b')}-z\ell\in E'_\beta$.  
But then 
\[ \tilde{v}^{(i',b')} = v^{(i',b')}-z\ell = \tilde{v}^{(i,b)}. \]
\end{proof}

\section{Some elementary formulas}

In this section we record formulas for certain constants that will appear as coefficients of the formal logarithmic series solutions in the next section.
Let $v\in{\mathbb C}$ and let $r\in{\mathbb Z}_{\geq 0}$.  There is a unique sequence of functions $\{f_l^{(v,r)}(t)\}_{l\in{\mathbb Z}}$ of one variable $t$ such that $f_0^{(v,r)}(t) = t^v\log^r t$, $d/dt\big(f_l^{(v,r)}(t)\big) = f_{l-1}^{(v,r)}(t)$ for all $l\in{\mathbb Z}$, and $f_l^{(v,r)}(t)$ equals $t^{v+l}$ times a polynomial in $\log t$ with constant coefficients.  For $l<0$ these functions are obtained by $-l$ differentiations of $t^v\log^r t$.  For $l>0$ they are obtained by $l$ integrations of $t^v\log^r t$, taking the constant of integration to be $0$ each time.

To give explicit formulas for some of the $f^{(v,r)}_l(t)$, we define expressions $M_{l,s}(v)$ for $l\in{\mathbb Z}$ and an integer $s\geq 0$.  First of all, for $l=0$ set
\begin{equation}
M_{0,s}(v) = \begin{cases} 1 & \text{if $s=0$,} \\ 0 & \text{if $s>0$.} \end{cases} 
\end{equation}
For $l>0$ (and $v+l<0$ if $v\in{\mathbb Z}_{<0}$), set
\begin{equation}
M_{l,s}(v) = \frac{(-1)^s}{(v+1)\cdots(v+l)} \sum_{i_1+\cdots+i_l=s}\frac{1}{(v+1)^{i_1}}\cdots \frac{1}{(v+l)^{i_l}}. 
\end{equation}
For $l<0$, set
\begin{equation}
M_{l,s}(v) = \begin{cases} S^{(-l)}_{-l-s}(v,v-1,\dots,v+l+1) & \text{if $s\leq-l$,} \\ 0 & \text{if $s>-l$.}
\end{cases} 
\end{equation}
where $S^{(\sigma)}_\tau(x_1,\dots,x_\sigma)$ is the elementary symmetric polynomial of degree $\tau$ in $\sigma$ variables.  (Note: $S^{(\sigma)}_0(x_1,\dots,x_\sigma)=1$, so $M_{l,-l}(v) = 1$.)

The following result is straightforward to check by induction on $l$.
\begin{lemma}
The functions $f_l^{(v,r)}(t)$ are given by
\begin{equation}
f_l^{(v,r)}(t) = t^{v+l}\sum_{s=0}^r M_{l,s}(v) r(r-1)\cdots(r-s+1)\log^{r-s}t
\end{equation}
for all $l$ if $v\not\in{\mathbb Z}_{<0}$ and for all $l$ such that $v+l<0$ if $v\in{\mathbb Z}_{<0}$.  Equivalently, Equation~$(4.5)$ {\em fails\/} to give the correct expression for $f_l^{(v,r)}(t)$ if and only if $v\in{\mathbb Z}_{<0}$ and $v+l\in{\mathbb Z}_{\geq 0}$.  
\end{lemma}

In the case where $v\in{\mathbb Z}_{<0}$ and $v+l\in{\mathbb Z}_{\geq 0}$ the function $f_l^{(v,r)}(t)$ equals $t^{v+l}$ times a polynomial of degree $r+1$ in $\log t$ with constant coefficients.  These functions will play no role in our construction of formal logarithmic solutions so we do not give explicit formulas for them.

For convenient reference in Sections 6 and 7 we record the values of $M_{l,s}(v)$ for $s=0,1$.  Recall the Pochhammer symbol: for a nonnegative integer~$l$, 
\[ (v)_l= v(v+1)\cdots(v+l-1) . \]
For $l=0$,
\begin{equation}
M_{0,0}(v) = 1\quad\text{and}\quad M_{0,1}(v) = 0.
\end{equation}
For $l>0$ (and $v+l<0$ if $v\in{\mathbb Z}_{<0}$),
\begin{equation}
 M_{l,0}(v) = \frac{1}{(v+1)_l} 
 \end{equation}
and 
\begin{equation}
M_{l,1}(v) = -\frac{1}{(v+1)_l}\bigg( \frac{1}{v+1} + \frac{1}{v+2} + \cdots + \frac{1}{v+l}\bigg). 
\end{equation}
For $l<0$,
\begin{equation}
M_{l,0}(v) = v(v-1)\cdots(v+l+1) = (-1)^{-l}(-v)_{-l} 
\end{equation}
and
\begin{align}
 M_{l,1}(v) &= \sum_{i=0}^{-l-1} v(v-1)\cdots\widehat{(v-i)}\cdots (v+l+1) \\
  & = (-1)^{-l}(-v)_{-l} \sum_{i=0}^{-l-1} \frac{1}{v-i},
\end{align}
where (4.11) is valid for $v\not\in\{0,1,\dots,-l-1\}$.  (Note: $M_{-1,1}(v) = 1$.)

\section{Formal logarithmic solutions}

The results of this section apply to more general $A$-hypergeometric systems than those described in Section~1.  Specifically, we no longer restrict to the codimension-one case.  For that reason the notation of this section differs from the remainder of the paper.  We reimpose our original notation and hypotheses after this section.  Let $A = \{{\bf a}_1,\dots,{\bf a}_n\}\subseteq{\mathbb Z}^d$ and let $\beta\in{\mathbb C}^d$.  

We describe the system of partial differential equations associated to $A$ and $\beta$. The lattice of relations $L$ on the set $A$ is defined to be
\[ L = \bigg\{ l=(l_1,\dots,l_n)\in{\mathbb Z}^d\mid \sum_{j=1}^n l_j{\bf a}_j = {\bf 0}\bigg\}. \]
To each $l\in L$ there is associated a box operator
\begin{equation}
 \Box_l = \prod_{l_j>0} \partial_j^{l_j} - \prod_{l_j<0} \partial_j^{-l_j}.  
 \end{equation}
For ${\bf a}_j\in A$ write ${\bf a}_j = (a_{1j},\dots,a_{dj})$ and let $\beta=(\beta_1,\dots,\beta_d)\in {\mathbb C}^d$.  Associated to $\beta$ and the set $A$ are the Euler (or homogeneity) operators
\begin{equation}
 \sum_{j=1}^n a_{ij}x_j\partial_j - \beta_i \text{ for $i=1,\dots,d$.} 
 \end{equation}
 The \emph{A-hypergeometric system with parameter $\beta$\/} is the system consisting of the box operators (5.1) for $l\in L$ and the Euler operators~(5.2).
 
In this section, we give a method for constructing formal logarithmic solutions to the $A$-hypergeometric systems with parameter $\beta$.  This method was outlined in \cite{AS1}; here we provide explicit formulas for higher-order log solutions.  In the next section, we apply this method in the setting of Section 1 and determine which of these formal solutions belong to a Nilsson ring.  

Let $v=(v_1,\dots,v_n)\in{\mathbb C}^n$ satisfy $\sum_{i=1}^n v_i{\bf a}_i = \beta$.  
For a nonnegative integer $r$, put ${\mathcal P}_r=\{1,\dots,n\}^r$, the set of sequences of length $r$ from the set $\{1,\dots,n\}$.  (In the special case $r=0$, we define ${\mathcal P}_0 = \{\emptyset\}$, i.~e., the only sequence of length $0$ is the empty sequence.)  Let $P = (p_1,\dots,p_r) \in{\mathcal P}_r$.  Up to ordering, the sequence $P$ is determined by its associated multiplicity function $\rho_P:\{1,\dots,n\}\to{\mathbb N}$ defined by
\[ \rho_P(i) = {\rm card}\{j\mid p_j=i\}. \]
(Thus for $r=0$ we have $P=\emptyset$ and $\rho_P(i) = 0$ for all $i$.)  

We associate to $v$ and $P$ a generating series $\Psi_v^P(x,T)$ (where $T^c=T_1^{c_1}\cdots T_d^{c_d}$ for $c=(c_1,\dots,c_d)\in{\mathbb C}^d$) defined using the functions $f_l^{(v,r)}(t)$ defined in Section~4.  First put 
\begin{equation}
\Psi^{(\rho_P(i))}_{v_i}(x_i,T) = \sum_{l_i\in{\mathbb Z}} f^{(v_i,\rho_P(i))}_{l_i}(x_i)T^{(v_i+l_i){\bf a}_i},
\end{equation}
then set
\begin{equation}
\Psi_v^P(x,T) = \prod_{i=1}^n \Psi_{v_i}^{(\rho_P(i))}(x_i,T).
\end{equation}
This series depends only on $\rho_P$, not on $P$, i.~e., it is independent of the ordering of the sequence~$P$.  But for bookkeeping purposes, it is useful to index it by $P$.  We use these series to construct formal logarithmic solutions of the $A$-hypergeometric system.

Write 
\begin{equation}
\Psi_v^P(x,T) = \sum_{u\in{\mathbb Z}A} \Psi^P_{v,\beta+u}(x)T^{\beta+u}.
\end{equation}
Equation (5.3) and the definition of the $f^{(v_i,\rho_P(i))}_{l_i}(x_i)$ imply that
\[ \partial_i \Psi_{v_i}^{(\rho_P(i))}(x_i,T) = T^{{\bf a}_i}\Psi_{v_i}^{(\rho_P(i))}(x_i,T), \]
hence by (5.4)
\[ \partial_i \Psi_v^P(x,T) = T^{{\bf a}_i}\Psi_{v}^P(x,T), \]
or, equivalently, from (5.5)
\begin{equation}
\partial_i\Psi^P_{v,\beta+u}(x) = \Psi^P_{v,\beta+u-{\bf a}_i}(x).
\end{equation}

\begin{proposition}
For all $u\in{\mathbb Z}A$, the expression $\Psi^P_{v,\beta+u}(x)$ satisfies the box operators $\Box_l$ for $l\in L$.
\end{proposition}

\begin{proof}
Let $l=(l_1,\dots,l_n)\in L$.  Equation (5.6) implies that
\begin{equation}
\Box_l(\Psi^P_{v,\beta+u}(x)) = \Psi^P_{v,\beta+u-\sum_{l_j>0}l_j{\bf a}_j}(x)-\Psi^P_{v,\beta+u+\sum_{l_j<0} l_j{\bf a}_j}(x).
\end{equation}
But $l\in L$ implies that $\sum_{l_j>0}l_j{\bf a}_j =-\sum_{l_j<0} l_j{\bf a}_j$, so the right-hand side of (5.8) vanishes.
\end{proof}

In general the $\Psi^P_{v,\beta+u}(x)$ do not satisfy the Euler operators because of the presence of logarithms. We need to impose some further conditions to construct solutions.
Let $v\in{\mathbb C}^n$ and let $I\subseteq\{1,\dots,n\}$.  We define the {\it $I$-negative support of $v$\/} to be
\[ \text{$I$-nsupp}(v) = \{ i\in I\mid v_i\in{\mathbb Z}_{<0}\}. \]
Let $u\in{\mathbb Z}A$ and put
\[ L(u) = \bigg\{l=(l_1,\dots,l_n)\in{\mathbb Z}^n\mid \sum_{j=1}^n l_j{\bf a}_j = u\bigg\}, \]
so $L=L({\bf 0})$.  
We say that $v$ has {\it minimal $(I,u)$-negative support\/} if $\text{$I$-nsupp}(v+l)$ is not a proper subset of $\text{$I$-nsupp}(v)$ for any $l\in L(u)$.  For example, to say that $v$ has minimal negative support is equivalent to saying that $v$ has minimal $(\{1,\dots,n\},{\bf 0})$-negative support.  We define
\begin{equation}
 L_{v,I}(u) = \{ l\in L(u)\mid \text{$I$-nsupp$(v+l) = I$-nsupp$(v)$}\}. 
 \end{equation}
If $I\subseteq I'$, then $I$-nsupp$(v)\subseteq I'$-nsupp$(v)$, so $L_{v,I'}(u)\subseteq L_{v,I}(u)$.

From (5.3), (5.4), and (5.5) we compute the formal series
\begin{align}
\Psi^P_{v,\beta+u}(x) &= \sum_{l\in L(u)} \prod_{i=1}^n f^{(v_i,\rho_P(i))}_{l_i}(x_i)\\ \nonumber
&\in {\mathbb C}[\log x_1,\dots,\log x_n][[x_1,x_1^{-1},\dots,x_n,x_n^{-1}]].
\end{align}
Let $\bar{P}=\{i\in\{1,\dots,n\}\mid \rho_P(i)>0\}$.  For a subset $I\subseteq\{1,\dots,n\}$ we denote its complement by $I^{\rm c}$.
\begin{lemma}
Suppose that $v$ has minimal $(\bar{P}^{\rm c},u)$-negative support.  Then
\begin{equation}
\Psi^P_{v,\beta+u}(x) = \sum_{l\in L_{v,\bar{P}^{\rm c}}(u)} \prod_{i=1}^n f^{(v_i,\rho_P(i))}_{l_i}(x_i).
\end{equation}
\end{lemma}

\begin{proof}
Let $l\in L(u)$.  Since $v$ has minimal $(\bar{P}^{\rm c},u)$-negative support, if $l\not\in L_{v,\bar{P}^{\rm c}}(u)$, i.~e., if $\bar{P}^{\rm c}$-nsupp$(v+l)\neq\bar{P}^{\rm c}$-nsupp$(v)$, then there exists an index $i\in\bar{P}^{\rm c}$ such that $v_i+l_i\in{\mathbb Z}_{<0}$ but $v_i\in{\mathbb Z}_{\geq 0}$.  Since $\rho_P(i) = 0$, the product on the right-hand side of (5.10) contains the factor 
$f_{l_i}^{(v_i,0)}(x_i)$.  But this factor equals $0$ by Lemma~4.4 and (4.9):
\[ f_{l_i}^{(v_i,0)}(x_i) = M_{l_i,0}(v_i)x_i^{v_i+l_i} = v_i(v_i-1)\cdots (v_i+l_i+1)x_i^{v_i+l_i} = 0.  \]
Equation~(5.12) then follows from (5.10).
\end{proof}

\begin{proposition}
Suppose that $v$ has minimal $(\bar{P}^{\rm c},u)$-negative support and also minimal $(\bar{P}^{\rm c}\cup\{j\},u)$-negative support for all $j\in \bar{P}$.  Then all factors $f^{(v_i,\rho_P(i))}_{l_i}(x_i)$ appearing on the right-hand side of 
$(5.12)$ are given by $(4.5)$.
\end{proposition}

\begin{proof}
Let $i\in\bar{P}^{\rm c}$.  There does not exist $l\in L_{v,\bar{P}^{\rm c}}(u)$ with $v_i\in{\mathbb Z}_{<0}$ and $v_i+l_i\in{\mathbb Z}_{\geq 0}$ because $\bar{P}^{\rm c}\text{-nsupp}(v+l) = \bar{P}^{\rm c}\text{-nsupp}(v)$.
This implies by Lemma~4.4 that if $\rho_P(i)=0$, then the factor $f_{l_i}^{(v_i,0)}(x_i)$ on the right-hand side of (5.12) is given by~(4.5).

Let $j\in \bar{P}$.  There is no $l\in L_{v,\bar{P}^{\rm c}}(u)$ with $v_j\in{\mathbb Z}_{<0}$ and $v_j+l_j\in{\mathbb Z}_{\geq 0}$, otherwise $(\bar{P}^{\rm c}\cup\{j\})$-nsupp$(v+l)$ would be a proper subset of $(\bar{P}^{\rm c}\cup\{j\})$-nsupp$(v)$, contradicting the assumption that $v$ has minimal $(\bar{P}^{\rm c}\cup\{j\},u)$-negative support.  This implies by Lemma~4.4 that if $\rho_P(j)>0$, then the factor $f_{l_j}^{(v_j,\rho_P(j))}(x_j)$ on the right-hand side of (5.12) is given by~(4.5).
\end{proof}

When the hypothesis of Proposition 5.13 is satisfied, we compute from (5.12) and~(4.5)
\begin{multline}
\Psi^P_{v,\beta+u}(x) = \sum_{l\in L_{v,\bar{P}^{\rm c}}(u)} x_1^{v_1+l_1}\cdots x_n^{v_n+l_n}  \\ 
\cdot\prod_{i=1}^n \bigg(\sum_{j_i=0}^{\rho_P(i)} M_{l_i,j_i}(v_i) \rho_P(i)(\rho_P(i)-1)\cdots(\rho_P(i)-j_i+1) \log^{\rho_P(i)-j_i}x_i\bigg). 
\end{multline}

Let $\rho,\rho':\{1,\dots,n\}\to{\mathbb N}$ be arbitrary functions.  We write $\rho'\leq\rho$ if $\rho'(i)\leq\rho(i)$ for all~$i$.  Put
\[ R(\rho_P) = \{ \rho:\{1,\dots,n\}\to{\mathbb N}\mid \rho\leq\rho_P\}. \]
We may rewrite (5.14) as
\begin{multline}
\Psi^P_{v,\beta+u}(x) = \sum_{l\in L_{v,\bar{P}^{\rm c}}(u)} x_1^{v_1+l_1}\cdots x_n^{v_n+l_n}  \\
\cdot\sum_{\rho\in R(\rho_P)} \prod_{i=1}^n M_{l_i,\rho(i)}(v_i) \rho_P(i)(\rho_P(i)-1)\cdots(\rho_P(i)-\rho(i)+1)\frac{\log^{\rho_P(i)}x_i}{\log^{\rho(i)}x_i}.
\end{multline}

The set $R(\rho_P)$ is related to subsequences of $P$.  For $0\leq k\leq r$, let $S(P,k)$ denote the set of subsequences of $P$ of length $k$.  For $Q\in S(P,k)$ we clearly have $\rho_Q\leq \rho_P$.  Conversely, if we are given $\rho\in R(\rho_P)$, then setting $k=\sum_{i=1}^n \rho(i)$, there exist $Q\in S(P,k)$ such that $\rho_Q=\rho$.  Furthermore, it is easy to see that for this $\rho$ we have
\[ {\rm card}\{Q\in S(P,k)\mid \rho_Q=\rho\} = \prod_{i=1}^n \rho_P(i)(\rho_P(i)-1)\cdots(\rho_P(i)-\rho(i)+1). \]
We can thus rewrite (5.15) as
\begin{multline}
\Psi^P_{v,\beta+u}(x) = \sum_{l\in L_{v,\bar{P}^{\rm c}}(u)} x_1^{v_1+l_1}\cdots x_n^{v_n+l_n} \\
\cdot\sum_{k=0}^r \sum_{Q=(p_{j_1},\dots,p_{j_k})\in S(P,k)} \bigg(\prod_{i=1}^n M_{l_i,\rho_Q(i)}(v_i)\bigg) \frac{\log x_{p_1}\cdots\log x_{p_r}} {\log x_{p_{j_1}}\cdots\log x_{p_{j_k}}}.
\end{multline}

For $Q\in S(P,k)$, we set
\begin{equation}
\Phi^Q_{v,\beta+u}(x) = \sum_{l\in L_{v,\bar{P}^{\rm c}}(u)} \bigg(\prod_{i=1}^n M_{l_i,\rho_Q(i)}(v_i)\bigg)x_1^{v_1+l_1}\cdots x_n^{v_n+l_n}.
\end{equation}

From Equation (5.16) we have the following result.
\begin{corollary}
Under the hypothesis of Proposition $5.13$ we have
\begin{equation}
\Psi^P_{v,\beta+u}(x) = \sum_{k=0}^r \sum_{Q=(p_{j_1},\dots,p_{j_k})\in S(P,k)} \Phi^Q_{v,\beta+u}(x)\frac{\log x_{p_1}\cdots\log x_{p_r}} {\log x_{p_{j_1}}\cdots\log x_{p_{j_k}}},
\end{equation}
where the $\Phi_{v,\beta+u}^Q(\Lambda)$ are given by $(5.17)$.
\end{corollary}

The series $\Phi^Q_{v,\beta+u}(x)$ depends on $\rho_Q$.  It is independent of the ordering of $Q$, but it also depends on the choice of $P$ for which $Q\in S(P,k)$ because the sum in (5.17) is over $L_{v,\bar{P}^{\rm c}}(u)$.
We impose an additional condition to eliminate the dependence on $P$.
\begin{lemma}
Suppose that $v$ has minimal $(\bar{Q}^{\rm c},u)$-negative support.  Then
\begin{equation}
\Phi^Q_{v,\beta+u}(x) = \sum_{l\in L_{v,\bar{Q}^{\rm c}}(u)} \bigg(\prod_{i=1}^n M_{l_i,\rho_Q(i)}(v_i)\bigg)x_1^{v_1+l_1}\cdots x_n^{v_n+l_n}.
\end{equation}
\end{lemma}

\begin{proof}
Since $\bar{Q}\subseteq\bar{P}$ we have $\bar{P}^{\rm c}\subseteq\bar{Q}^{\rm c}$ and $L_{v,\bar{Q}^{\rm c}}(u) \subseteq L_{v,\bar{P}^{\rm c}}(u)$.  Suppose that $l\in L_{v,\bar{P}^{\rm c}}(u)$ but $l\not\in L_{v,\bar{Q}^{\rm c}}(u)$.  Then
\[ \bar{Q}^{\rm c}\text{-nsupp}(v+l)\neq \bar{Q}^{\rm c}\text{-nsupp}(v). \]
Since $v$ has minimal $(\bar{Q}^{\rm c},u)$-negative support, there exists $i\in\bar{Q}^{\rm c}$ such that $v_i+l_i\in{\mathbb Z}_{< 0}$ but $v_i\in{\mathbb Z}_{\geq 0}$.  But then $\rho_Q(i)=0$ and $M_{l_i,0}(v_i)=0$ since $v_i\in{\mathbb Z}_{\geq 0}$ and $v_i+l_i\in{\mathbb Z}_{<0}$ (see Equation (4.9)).  This shows that the terms in the sum on the right-hand side of (5.17) vanish for $l\in L_{v,\bar{P}^{\rm c}}(u)\setminus L_{v,\bar{Q}^{\rm c}}(u)$.  
\end{proof}

\begin{theorem}
Suppose that $v$ has minimal $(I,u)$-negative support for all $I\subseteq\{1,\dots,n\}$ with $\lvert I\rvert\geq n-r$
and let $l^{(k)}=(l^{(k)}_1,\dots,l^{(k)}_n)\in L$ for $k=1,\dots,r$.  Then the expression 
\begin{equation}
\sum_{P=(p_1,\dots,p_r)\in{\mathcal P}_r} l^{(1)}_{p_1}\cdots l^{(r)}_{p_r} \Psi^P_{v,\beta+u}(x) 
\end{equation}
is a solution of the $A$-hypergeometric system with parameter $\beta+u$.
\end{theorem}

\begin{proof}
The expression (5.23) satisfies the box operators because it is a linear combination of expressions $\Psi^P_{v,\beta+u}(x)$ that satisfy the box operators (Proposition~5.7).  We need to prove that it satisfies the Euler operators with parameter~$\beta+u$.

The hypothesis of the theorem implies that the hypotheses of Proposition 5.13 and Lemma 5.20 are satisfied for all $P\in{\mathcal P}_r$ and all $Q\in S(P,k)$.  Thus the $\Psi^P_{v,\beta+u}(x)$ are given by (5.19) and the $\Phi^Q_{v,\beta+u}(x)$ are given by (5.21).  

Substitute (5.19) into (5.23):
\begin{multline}
\sum_{P=(p_1,\dots,p_r)\in{\mathcal P}_r} l^{(1)}_{p_1}\cdots l^{(r)}_{p_r} \\
\cdot\sum_{k=0}^r \sum_{Q=(p_{j_1},\dots,p_{j_k})\in S(P,k)} \Phi^Q_{v,\beta+u}(x)\frac{\log x_{p_1}\cdots\log x_{p_r}}{{\log x}_{p_{j_1}}\cdots{\log x}_{p_{j_k}}}.
\end{multline}
We can reverse the order of summation because, by Lemma 5.20, $\Phi^Q_{v,\beta+u}(x)$ is independent of the choice of $P$ for which $Q\in S(P,k)$.  For a sequence $Q=(p_{j_1},\dots,p_{j_k})\in{\mathcal P}_k$, the set of all sequences $P=(p_1,\dots,p_r)\in{\mathcal P}_r$ for which $Q\in S(P,k)$ has cardinality $n^{r-k}$: the values $p_j$ for $j\not\in\{{j_1},\dots,{j_k}\}$ can be assigned arbitrarily from $\{1,\dots,n\}$.  Expression (5.24) thus equals
\begin{multline}
\sum_{k=0}^r \sum_{1\leq j_1<\dots<j_k\leq r}\sum_{{p_{j_1}},\dots,{p_{j_k}} = 1}^n l^{(j_1)}_{p_{j_1}}\cdots l^{(j_k)}_{p_{j_k}}\Phi^{(p_{j_1},\dots,p_{j_k})}_{v,\beta+u}(x) \\
\cdot\sum_{p_1,\dots\hat{p}_{j_1},\dots,\hat{p}_{j_k}\dots,p_r=1}^n l^{(1)}_{p_1}\cdots\widehat{l^{(j_1)}_{p_{j_1}}}\cdots\widehat{l^{(j_k)}_{p_{j_k}}}\cdots l^{(r)}_{p_r} \\ \cdot\log x_{p_1}\cdots\widehat{\log x}_{p_{j_1}}\cdots\widehat{\log x}_{p_{j_k}} \cdots\log x_{p_r}.
\end{multline}
Note that the innermost sum in (5.25) equals
\[ \prod_{i=1,\dots,\hat{\jmath}_1,\dots,\hat{\jmath}_k,\dots,r} \log x^{l^{(i)}}, \]
so (5.25) simplifies to 
\begin{equation}
\sum_{k=0}^r \sum_{1\leq j_1<\dots<j_k\leq r}\sum_{{p_{j_1}},\dots,{p_{j_k}} = 1}^n l^{(j_1)}_{p_{j_1}}\cdots l^{(j_k)}_{p_{j_k}}\Phi^{(p_{j_1},\dots,p_{j_k})}_{v,\beta+u}(x) \prod_{\substack{i=1,\dots,r\\ i\neq j_1,\dots,j_k}} \log x^{l^{(i)}}. 
\end{equation}
Since $l\in L(u)$, every monomial $x^{v+l}$ appearing in some $\Phi^Q_{v,\beta+u}(x)$ satisfies the Euler operators with parameter $\beta+u$, which implies by a straightforward calculation that $x^{v+l}\prod_{\substack{i=1,\dots,r\\ i\neq j_1,\dots,j_k}} \log x^{l^{(i)}}$ also satisfies the Euler operators with parameter $\beta+u$.  It follows that the expression (5.26) (and hence (5.23), which it equals) satisfies the Euler operators with parameter $\beta+u$.
\end{proof}

Note that we have obtained two formulas for the solution.  One, in terms of the $\Psi_{v,\beta+u}^P(x)$, is given in (5.23) and the other, in terms of the $\Phi^Q_{v,\beta+u}(x)$, is given in (5.26).  

There is one obvious case where $v$ has minimal $(I,u)$-negative support for all $I\subseteq\{1,\dots,n\}$ and all $u\in{\mathbb Z}A$, namely, the case where $v_i\not\in{\mathbb Z}_{<0}$ for $i=1,\dots,n$.  In Section~7 we show that in the codimension-one case, if $\beta$ is nonresonant, then every $v\in E'_\beta$ satisfies this condition.

\section{Logarithmic series solutions}

We return to the notation of Section 1: Let $A$ and $\beta$ be as in Section~1 and let $v=(v_1,\dots,v_n)\in E'_\beta$.  Fix $r$, $0\leq r<m_v$, and $u\in{\mathbb Z}A$.  We make the assumption that $v$ has minimal $(I,u)$-negative support for all $I\subseteq\{1,\dots,n\}$ with $\lvert I\rvert\geq n-r$.  

The lattice $L$ of relations on $A$ is the set $L=\{z\ell\mid z\in{\mathbb Z}\}$.  We need to describe the sets $L_{v,I}(u)$.  Since $u\in{\mathbb Z}A$, we can choose $\ell^{(u)}\in{\mathbb Z}^n$ such that $\sum_{\sigma=1}^n \ell^{(u)}_\sigma{\bf a}_\sigma = u$.  Then $L(u) = \{\ell^{(u)}+z\ell\mid z\in{\mathbb Z}\}$.  
Every element of $L_{v,I}(u)$ is thus of the form $\ell^{(u)} + z\ell$ for some $z\in{\mathbb Z}$.  Put
\begin{equation}
{\mathbb Z}_{v,I}(u,\ell^{(u)}) = \{z\in{\mathbb Z}\mid \ell^{(u)}+z\ell\in L_{v,I}(u)\},
\end{equation}
i.~e., ${\mathbb Z}_{v,I}(u,\ell^{(u)})$ is the set of all $z\in{\mathbb Z}$ such that
\[ \text{$I$-nsupp}(v+\ell^{(u)}+z\ell) = \text{$I$-nsupp}(v). \]
By Lemma 5.20, Equation~(5.21) gives us the formula for
$\Phi^Q_{v,\beta+u}(x)$:
\begin{multline}
\Phi^Q_{v,\beta+u}(x) = \\ 
x^{v+\ell^{(u)}} \sum_{z\in{\mathbb Z}_{v,\bar{Q}^{\rm c}}(u,\ell^{(u)})} \bigg(\prod_{i=1}^k M_{\ell^{(u)}_i+z\ell_i,\rho_Q(i)}(v_i)\prod_{j=k+1}^nM_{\ell^{(u)}_j-z\ell_j,\rho_Q(j)}(v_j)\bigg)x_0^z.
\end{multline}
Since all elements of $L$ are scalar multiples of $\ell$, all choices of $l^{(k)}$ in Theorem~5.22 must be scalar multiples of $\ell$ also.  Thus all possibilities arising from (5.23) are scalar multiples of the result obtained by choosing $l^{(k)}=\ell$ for $k=0,1,\dots,r$.  Equation (5.26) then gives the formal logarithmic solution of degree $r$ with parameter $\beta+u$ associated to $v\in E'_{\beta}$:
\begin{multline}
\sum_{s=0}^r \sum_{1\leq t_1<\dots<t_s\leq r} \sum_{p_{t_1},\dots,p_{t_s}=1}^n \\ 
\bigg(\prod_{1\leq p_{t_s}\leq k} \ell_{p_{t_s}}\prod_{k+1\leq p_{t_s}\leq n} (-\ell_{p_{t_s}}) \bigg)\Phi_{v,\beta+u}^{(p_{t_1},\dots,p_{t_s})}(x) \log^{r-s}x_0.
\end{multline}

Note that when $r=0$ this expression reduces to $\Phi_{v,\beta+u}^\emptyset(x)$, which is the formal logarithm-free series solution with parameter $\beta+u$ associated to $v\in E'_\beta$.
We make explicit the case $Q=\emptyset$ in (6.2).   From (6.2) we have
\begin{multline}
 \Phi^\emptyset_{v,\beta+u}(x) = \\ 
 x^{v+\ell^{(u)}}\sum_{z\in{\mathbb Z}_{v,\{1,\dots,n\}}(u,\ell^{(u)})} \bigg(\prod_{i=1}^k M_{\ell^{(u)}_i+z\ell_i,0}(v_i) \prod_{j=k+1}^n M_{\ell^{(u)}_i-z\ell_j,0}(v_j)\bigg)x_0^z. 
\end{multline}

{\bf Example.}  By Lemma~3.2 every $v\in E'_\beta$ has minimal $(\{1,\dots,n\},{\bf 0})$-negative support.  By the definition of $E'_\beta$ only the $v_j$ for $j=k+1,\dots,n$ can be negative integers.  Suppose that $v_j\in{\mathbb Z}_{<0}$ for some $j$.  Then for $z\in{\mathbb Z}$, one has $v_j-z\ell_j\in{\mathbb Z}_{\geq 0}$ only if $z<0$.  But then some $v_i+z\ell_i\in{\mathbb Z}_{<0}$ by the definition of $E'_\beta$.  Take $u={\bf 0}$ and $\ell^{(u)}={\bf 0}$ in~(6.1), so ${\mathbb Z}_{v,\{1,\dots,n\}}({\bf 0},{\bf 0})$ is the set of $z\in{\mathbb Z}$ such that
\[ \text{$\{1,\dots,n\}$-nsupp}(v+z\ell) = \text{$\{1,\dots,n\}$-nsupp}(v). \]
Then this argument shows that
\[ 0\in{\mathbb Z}_{v,\{1,\dots,n\}}({\bf 0},{\bf 0})\subseteq {\mathbb Z}_{\geq 0}. \]
Equation (6.4) then gives
\[  \Phi^\emptyset_{v,\beta}(x) = 
 x^{v}\sum_{z\in{\mathbb Z}_{v,\{1,\dots,n\}}({\bf 0},{\bf 0})} \bigg(\prod_{i=1}^k M_{z\ell_i,0}(v_i) \prod_{j=k+1}^n M_{-z\ell_j,0}(v_j)\bigg)x_0^z. \]
 In particular, the coefficient of $x^v$ on the right-hand side is 1, so $\Phi^\emptyset_{v,\beta}(x)\neq 0$.  Thus every $v\in E'_\beta$ gives a nonzero, logarithm-free solution to the $A$-hypergeometric system with parameter $\beta$ which belongs to the ring ${\mathbb C}[[x_0]][x^v]$.  

Our main result is the following theorem.
\begin{theorem}
Suppose that $v\in E'_\beta$ has minimal $(I,u)$-negative support for all $I\subseteq\{1,\dots,n\}$ with $\lvert I\rvert\geq n-r$ and some $r<m_v$.  Then the expression $(6.3)$ is a solution of the $A$-hypergeometric system with parameter $\beta+u$ that lies in the Nilsson ring ${\mathbb C}[[x_0]][x_0^{-1},\log x_0,x^v]$.
\end{theorem}

\begin{proof}
We already noted that (6.3) is a formal solution, we just need to check the Nilsson ring assertion.  For that purpose, we need to check that each $\Phi^Q_{v,\beta+u}(x)$, where $Q$ is a sequence of length less than or equal to $r$,  lies in the Nilsson ring.  
Since 
$\lvert \bar{Q}\rvert\leq r<m_v=\lvert I_v\rvert$ there exists $i\in I_v$ such that $i\in\bar{Q}^{\rm c}$.  Since 
$i\in I_v$ we have $v_i\in{\mathbb Z}_{\geq 0}$, and since $\ell^{(u)}+z\ell\in L_{v,\bar{Q}^{\rm c}}(u)$ we have
$v_i+\ell^{(u)}_i+z\ell_i\in{\mathbb Z}_{\geq 0}$ also.  This latter condition can be satisfied for only finitely many $z\in{\mathbb Z}_{<0}$, which implies by (6.2) that $\Phi^Q_{v,\beta+u}(x)\in{\mathbb C}[[x_0]][x_0^{-1},x^v]$.  
\end{proof}

{\bf Remark.}  When the hypothesis of the theorem is satisfied for some~$r$, $0\leq r<m_v$, then it is also satisfied for
all $r'$, $0\leq r'\leq r$.  Thus Equation~(6.3) gives solutions of the $A$-hyper\-ge\-o\-met\-ric system with parameter 
$\beta+u$ which are nominally polynomials of degrees $0,1,\dots,r$ in $\log x_0$ with coefficients in ${\mathbb C}[[x_0]][x_0^{-1},x^v]$.  Note that the coefficient of the highest power of $\log x_0$ in each of these solutions is $\Phi^\emptyset_{v,\beta+u}(x)$.  If this coefficient is nonzero, then Equation~(6.3) gives $r+1$ linearly independent solutions of the $A$-hypergeometric system with parameter~$\beta+u$.

\begin{corollary}
Suppose that each $v\in E'_\beta$ satisfies the hypothesis of Theorem $6.5$ with $r=m_v-1$ and suppose that each
$\Phi^\emptyset_{v,\beta+u}(x)$ is nonzero.  Then the series $(6.3)$ form a set of $\sum_{i=1}^k \ell_i$ linearly independent solutions of the $A$-hypergeometric system with parameter $\beta+u$ that lie in the Nilsson ring ${\mathbb C}[[x_0]][x_0^{-1},\log x_0,\{x^v\}_{v\in E'_\beta}]$.
\end{corollary}

\begin{proof}
By the above remark each $v\in E'_\beta$ gives $m_v$ linearly independent solutions.  Solutions corresponding to different $v$ are linearly independent because each contains a factor $x^v$.  As a result there is a total of $\sum_{v\in E'_\beta} m_v$ linearly independent solutions, and $\sum_{v\in E'_\beta} m_v=\sum_{i=1}^k \ell_i$ by~(1.4).
 \end{proof}
 
 {\bf Example.} (cf.\ \cite[Example 3.5.3]{SST}, \cite[Examples 3.1, 4.6, and 5.8]{S})  Let $A = \{{\bf a}_\mu\}_{\mu=1}^3\subseteq{\mathbb Z}^2$, where ${\bf a}_1 = (1,0)$, ${\bf a}_2 = (1,2)$, and ${\bf a}_3 = (1,1)$.  We have the relation ${\bf a}_1 + {\bf a}_2-2{\bf a}_3 = {\bf 0}$ and $x_0 = x_1x_2/x_3^2$.  Take $\beta = (10,8)$.  This gives $E_{(10,8)} = \{(2,0,8), (0,-2,12)\}$ and $E'_{(10,8)} = \{(2,0,8)\}$.  The vector $v=(2,0,8)$ has minimal $(I,u)$-negative support for all $I\subseteq\{1,2,3\}$ and all $u\in{\mathbb Z}A = {\mathbb Z}^2$ because none of its entries is a negative integer.  For this example we take $u=(0,0)$, $\ell^{(u)}=(0,0,0)$ and apply Corollary 6.6.  We have $m_{(2,0,8)} = 2$ so we can take $r=0$ to get a logarithm-free solution and $r=1$ to get a log solution, giving $2$ $(=\ell_1+\ell_2={\rm vol}(\Delta(A)))$ solutions for the parameter $\beta = (10,8)$ associated to the vector $v=(2,0,8)$.  
 
 From the definitions (5.9) and (6.1) we get
 \[ L_{(2,0,8),\{1,2,3\}}({\bf 0}) = \{z(1,1,-2)\mid z\in\{0,1,2,3,4\}\}, \]
and ${\mathbb Z}_{(2,0,8),\{1,2,3\}}({\bf 0},{\bf 0}) = \{0,1,2,3,4\}$.  From (6.4), (4.6), (4.7), and~(4.9), the logarithm-free solution is
\begin{multline*}
 \Phi_{(2,0,8),(10,8)}^\emptyset(x) = x_1^2x_3^8\sum_{z=0}^4 M_{z,0}(2)M_{z,0}(0)M_{-2z,0}(8)x_0^z  \\ 
 = x_1^2x_3^8\sum_{z=0}^4 \frac{(-8)_{2z}}{z!(3)_z}x_0^z \\
 = x_1^2x_3^8 + \frac{56}{3} x_1^3x_2x_3^6 + 70 x_1^4x_2^2x_3^4 + 56x_1^5x_2^3x_3^2 + \frac{14}{3} x_1^6x_2^4x_3.
\end{multline*}

To find the log solution we need to compute $\Phi_{(2,0,8),(0,8)}^{\{\mu\}}(x)$ for $\mu=1,2,3$ and apply (6.3).  From (5.9) we have
\begin{align*}
L_{(2,0,8),\{2,3\}}({\bf 0}) &= \{z(1,1,-2)\mid z\in\{0,1,2,3,4\}\}, \\
L_{(2,0,8),\{1,3\}}({\bf 0}) &= \{z(1,1,-2)\mid z\in\{-2,-1,0,1,2,3,4\}\}, \\
L_{(2,0,8),\{1,2\}}({\bf 0}) &= \{z(1,1,-2)\mid z\in{\mathbb Z}_{\geq 0}\},
\end{align*}
so
\begin{align*}
{\mathbb Z}_{(2,0,8),\{2,3\}}({\bf 0},{\bf 0}) &= \{0,1,2,3,4\}, \\
{\mathbb Z}_{(2,0,8),\{1,3\}}({\bf 0},{\bf 0}) &= \{-2,-10,1,2,3,4\}, \\
{\mathbb Z}_{(2,0,8),\{1,2\}}({\bf 0},{\bf 0}) &= {\mathbb Z}_{\geq 0}.
\end{align*}
From (6.2) we then have
\begin{align*}
 \Phi^{\{1\}}_{(2,0,8),(10,8)}(x) &= x_1^2x_3^8\sum_{z=0}^4 M_{z,1}(2)M_{z,0}(0)M_{-2z,0}(8)x_0^z, \\
 \Phi^{\{2\}}_{(2,0,8),(10,8)}(x) &= x_1^2x_3^8\sum_{z=-2}^4 M_{z,0}(2)M_{z,1}(0)M_{-2z,0}(8)x_0^z, \\
 \Phi^{\{3\}}_{(2,0,8),(10,8)}(x) &= x_1^2x_3^8\sum_{z=0}^\infty M_{z,0}(2)M_{z,0}(0)M_{-2z,1}(8)x_0^z.
 \end{align*}
 These sums can be made explicit from the formulas in Section 4.  By (6.3) the log solution is
 \[ \Phi_{(2,0,8),(10,8)}^\emptyset(x) \log x_0 + \Phi^{\{1\}}_{(2,0,8),(10,8)}(x) + \Phi^{\{2\}}_{(2,0,8),(10,8)}(x) -2\Phi^{\{3\}}_{(2,0,8),(10,8)}(x). \]
 
The `initial monomial' (in the sense of \cite{SST}) in this solution is the term corresponding to $z=-2$ in $\Phi^{\{2\}}_{(2,0,8),(10,8)}(x)$.  One computes from the formulas in Section 4 that this monomial is $-(5940)^{-1}x_2^{-2}x_3^{12}$.  This implies that the vector $(0,-2,12)$, which lies in $E_{(10,8)}$ but not in $E'_{(10,8)}$, is an exponent of this system.
  
 {\bf Example.}  Let $A=\{{\bf a}_\mu\}_{\mu=1}^3\subseteq{\mathbb Z}^2$, where ${\bf a}_1 = (1,0)$, ${\bf a}_2 = (0,1)$, and ${\bf a}_3 = (1,1)$.  We have the relation ${\bf a}_1 + {\bf a}_2 - {\bf a}_3 = {\bf 0}$ and $x_0 = x_1x_2/x_3$.  Take $\beta = (0,0)$.  This gives $E'_{(0,0)} = \{ (0,0,0)\}$.  The vector $v=(0,0,0)$ has minimal $(I,u)$-negative support for all $I\subseteq\{1,2,3\}$ and all $u\in{\mathbb Z}A={\mathbb Z}^2$ because none of its entries is a negative integer.  Take $u=(0,0)$, $\ell^{(u)} = (0,0,0)$.  Corollary~6.6 with $r=0,1$ then gives 2 $(=\ell_1 + \ell_2={\rm vol}(\Delta(A)))$ solutions for the parameter $\beta = (0,0)$
 associated to the vector $v=(0,0,0)$.  First of all, $L_{v,\{1,2,3\}}({\bf 0}) = \{(0,0,0)\}$, so ${\mathbb Z}_{v,\{1,2,3\}}({\bf 0},{\bf 0}) = \{0\}$ and from (6.4) and (4.6) we get for $r=0$ that one solution is $\Phi^\emptyset_{v,\beta}(x) =  M_{0,0}(0)^3=1$.  For the case $r=1$, we need to compute $\Phi^{\{\mu\}}_{v,\beta}(x)$ for $\mu=1,2,3$.  For $\mu=1$ we have $L_{v,\{2,3\}}({\bf 0}) = \{(0,0,0)\}$, so ${\mathbb Z}_{v,\{2,3\}}({\bf 0},{\bf 0})=\{0\}$ and from (6.2) and (4.6) we get
\[ \Phi^{\{1\}}_{v,\beta}(x) = M_{0,1}(0)M_{0,0}(0)^2 = 0. \]
One gets $\Phi^{\{2\}}_{v,\beta}(x)=0$ by an identical argument.  For $\mu=3$, on the other hand, we have
$L_{v,\{1,2\}}({\bf 0}) = \{ z\ell\mid z\in{\mathbb Z}_{\geq 0}\}$ so ${\mathbb Z}_{v,\{1,2\}}({\bf 0},{\bf 0}) = {\mathbb Z}_{\geq 0}$.  From (6.2), (4.6), and (4.10) we get
\[ \Phi^{\{3\}}_{v,\beta}(x) = \sum_{z=0}^\infty M_{z,0}(0)^2M_{-z,1}(0)x_0^z = \sum_{z=0}^\infty \frac{(-1)^{z-1}}{z\cdot z!}x_0^z. \]
From (6.3) we get the second solution:
\[ \Phi^\emptyset_{v,\beta}(x)\log x_0 + \Phi^{\{1\}}_{v,\beta}(x) + \Phi^{\{2\}}_{v,\beta}(x) - \Phi^{\{3\}}_{v,\beta}(x) = \log x_0 + \sum_{z=1}^\infty \frac{(-1)^z }{z\cdot z!}x_0^z. \]

If we keep $\beta=(0,0)$ and now take $u=(-1,-1)$, then $\Phi^\emptyset_{v,\beta + u}(x) = 0$ by (6.4), so the hypothesis of Corollary~6.6 is not satisfied and we do not get a full set of solutions at $(-1,-1)$ from choosing $\beta = (0,0)$.  To get the solutions in this case, we need to take $\beta = (-1,-1)$ and compute that $E'_{(-1,-1)} = \{(0,0,-1)\}$.  The vector $v=(0,0,-1)$ has minimal  $(I,{\bf 0})$-negative support for $I\subseteq\{1,2,3\}$ with $\lvert I\rvert\geq 2$.  
Applying Corollary~6.6 to this choice of $v$ then gives two independent solutions, one of them logarithmic, for $\beta = (-1,-1)$.

Now consider the case $\beta=(1,-1)$.  We have $E'_\beta = \{(2,0,-1)\}$, but $v=(2,0,-1)$ does not have $(\{1,3\},{\bf 0})$-minimal negative support.  Thus Theorem~6.5 gives only a single solution, corresponding to $r=0$.  In this case, the second solution comes from the choice $v=(0,-2,1)\in E_\beta$.  This $v$ has minimal $(\{1,2,3\},{\bf 0})$-negative support so gives a formal solution by Theorem~5.22.  Since $v\not\in E'_\beta$, we cannot apply Theorem~6.5 to conclude that this formal solution lies in the Nilsson ring, we need to make a direct calculation from (6.3) and (6.4).  We have
\[ L_{(0,-2,1),\{1,2,3\}}({\bf 0}) = \{ (1,-1,0), (0,-2,1)\} \]
so ${\mathbb Z}_{(0,-2,1),\{1,2,3\}}({\bf 0},{\bf 0}) = \{0,1\}$.  The associated formal solution from (6.3) (with $r=0$) and (6.4) is
\[  \Phi^\emptyset_{(0,-2,1),(1,-1)}(\Lambda) = x_2^{-2}x_3(1-x_0), \]
which clearly lies in the Nilsson ring.

In the next section, we show that these failures to obtain a set of $\sum_{i=1}^k \ell_i$ solutions for all $\beta + u$ from the set $E'_{\beta}$ do not occur when $\beta$ is nonresonant.  The hypothesis of Corollary~6.6 is always satisfied for such $\beta$.

\section{The nonresonant case}
 
 If $\beta$ is nonresonant, then every $v\in E'_\beta$ satisfies the hypothesis of Lemma 2.13, so we have the following result.
\begin{proposition}
Suppose that $\beta$ is nonresonant and that $v\in E'_\beta$.  Then $v_j\not\in{\mathbb Z}$ for $j=k+1,\dots,n$.
\end{proposition}

\begin{theorem}
If $\beta$ is nonresonant, then every $v\in E'_\beta$ satisfies the hypothesis of Corollary $6.6$ for all $u\in{\mathbb Z}A$.
\end{theorem}

\begin{proof}
It follows from Proposition 7.1 and the definition of $E'_\beta$ that for nonresonant $\beta$ no coordinate of any $v\in E'_\beta$ is a negative integer, hence for all $I\subseteq\{1,\dots,n\}$ and all $u\in{\mathbb Z}A$ the $(I,u)$-negative support of $v$ is minimal because it equals the emptyset.  To apply Corollary~6.6 to get a full set of solutions with parameter $\beta+u$, we need to check that $\Phi^\emptyset_{v,\beta+u}(x)\neq 0$.  

For $Q=\emptyset$ we have $\bar{Q}^{\rm c} = \{1,\dots,n\}$, so the sum in (6.4) is over those $z\in{\mathbb Z}$ for which no coordinate of $v+\ell^{(u)}+z\ell$ is a negative integer.  By Proposition~7.1 and the definition of $E'_\beta$, the only integer coordinates of $v+\ell^{(u)}+z\ell$ are those indexed by $I_v$.  Thus the sum in (6.4) is over those $z\in{\mathbb Z}$ for which
\[ z\geq z_0:=\max\{-(v_\mu+\ell^{(u)}_\mu)/\ell_\mu \mid \mu\in I_v\}. \]
We get
\begin{equation}
\Phi^\emptyset_{v,\beta+u}(x) = x^{v+\ell^{(u)}}\sum_{z\geq z_0} \bigg(\prod_{i=1}^k M_{\ell^{(u)}_i + z\ell_i,0}(v_i)\prod_{j=k+1}^n M_{\ell^{(u)}_j - z\ell_j,0}(v_j)\bigg)x_0^z.
\end{equation}
Since the $v_\mu$ for $\mu\not\in I_v$ are all nonintegral, the corresponding factors on the right-hand side of (7.3) are all nonzero.  For $\mu\in I_v$, the $v_\mu$ lie in ${\mathbb Z}_{\geq 0}$ and the factors $M_{\ell^{(u)}_\mu + z\ell_\mu,0}(v_\mu)$ are nonzero for $z\geq z_0$.  Thus all terms on the right-hand side of (7.3) are nonzero.
\end{proof}

For $\beta$ nonresonant we thus get $\sum_{i=1}^k \ell_i$ linearly independent logarithmic series solutions at the origin for every parameter $\beta+u$.

We illustrate by applying our results to Gauss' hypergeometric equation 
\[ x(1-x)y'' + (\sigma - (1+\theta_1+\theta_2)x)y' -\theta_1\theta_2 y = 0. \]
For generic values of the parameters $\theta_1,\theta_2,\sigma$, the holomorphic solution at the origin is given by the series
\[ {}_2F_1(\theta_1,\theta_2;\sigma;x) = \sum_{z=0}^{\infty} \frac{(\theta_1)_z(\theta_2)_z}{(\sigma)_z z!}x^z. \]

The $A$-hypergeometric analogue of this equation is obtained by taking
\[ {\bf a}_1 = (1,1,-1),\ {\bf a}_2 = (0,0,1),\ {\bf a}_3 = (1,0,0), \ {\bf a}_4 = (0,1,0) \]  
and taking $\beta = (-\theta_1,-\theta_2,\sigma-1)$.  (A convenient dictionary for translating between classical hypergeometric series and their $A$-hypergeometric counterparts is given in Dwork and Loeser \cite[Appendix]{DL}.)  We have ${\mathbb Z}A = {\mathbb Z}^3$ and the lattice of relations is generated by the equation
\begin{equation}
{\bf a}_1 + {\bf a}_2 - {\bf a}_3 -{\bf a}_4 = {\bf 0},
\end{equation}
so $x_0 = (x_1x_2)/(x_3x_4)$ and $\ell = (1,1,-1,-1)$.
The polytope $\Delta(A)$ has four facets containing the origin, lying in the planes $x_1=0$, $x_2=0$, $x_1+x_3=0$, and $x_2+x_3=0$.  The condition that $\beta$ be nonresonant thus means that $\theta_1$, $\theta_2$, $\theta_1-\sigma$, $\theta_2-\sigma$ are not integers.  We have $\ell_1+\ell_2 = 2\  (= {\rm vol}(A))$, so there will be 2 independent solutions.  

There are, in general, two elements $v^{(1,0)},v^{(2,0)}\in E_\beta$ since $k=2$ and $\ell_1=\ell_2=1$.  To compute $v^{(1,0)}$ we solve
\[ (-\theta_1,-\theta_2,\sigma-1) = 0{\bf a}_1 + \sum_{r=2}^4 c_2{\bf a}_r \]
for $c_2,c_3,c_4$ to get
\begin{equation}
v^{(1,0)} = (0,\sigma-1,-\theta_1,-\theta_2).
\end{equation}
A similar calculation shows that
\begin{equation}
v^{(2,0)} = (1-\sigma,0,\sigma-\theta_1-1,\sigma-\theta_2-1).
\end{equation}
When $\sigma=1$ we have $v^{(1,0)} = v^{(2,0)}$, so $E_\beta$ is a singleton in that case.

There are thus three possibilities for $E'_\beta$.  If $\sigma\not\in{\mathbb Z}$, then
\begin{equation}
E'_\beta=\{v^{(1,0)},v^{(2,0)}\}.
\end{equation}
If $\sigma\in{\mathbb Z}_{\geq 1}$, then
\begin{equation}
E'_\beta = \{v^{(1,0)}\}.
\end{equation}
If $\sigma\in{\mathbb Z}_{\leq 1}$, then
\begin{equation}
E'_\beta = \{v^{(2,0)}\}.
\end{equation}
(Note that if $\sigma = 1$, then $v^{(1,0)} = v^{(2,0)}$.)

When $\sigma\not\in{\mathbb Z}$, we get the two logarithm-free solutions by substituting (7.5) and (7.6) into (6.4):
\begin{equation}
\Phi^\emptyset_{v^{(1,0)},\beta}(x) = x_2^{\sigma-1}x_3^{-\theta_1}x_4^{-\theta_2} \sum_{z=0}^\infty \frac{(\theta_1)_z(\theta_2)_z}{(\sigma)_z z!} x^z_0
\end{equation}
and
\begin{equation}
 \Phi_{v^{(2,0)},\beta}^\emptyset(x) = x_2^{\sigma-1}x_3^{-\theta_1}x_4^{-\theta_2}x_0^{1-\sigma}\sum_{z=0}^\infty \frac{(\theta_1-\sigma+1)_z(\theta_2-\sigma+1)_z}{(2-\sigma)_z z!}x_0^z.
 \end{equation}
 
When $\sigma\in{\mathbb Z}$ we are in the situation of (7.8) or (7.9).  We find the solutions in the case $\sigma\in{\mathbb Z}_{\geq 1}$, the calculation of solutions in the other case is similar.  In this case $v^{(1,0)}$ is given by (7.5), so the logarithm-free solution is given by (7.10).  By (6.3) the other solution is
 \begin{equation}
 \Phi_{v^{(1,0)},\beta}^\emptyset(x)\log x_0 + \Phi^{\{1\}}_{v^{(1,0)},\beta}(x)+\Phi^{\{2\}}_{v^{(1,0)},\beta}(x)-\Phi^{\{3\}}_{v^{(1,0)},\beta}(x)-\Phi^{\{4\}}_{v^{(1,0)},\beta}(x). 
 \end{equation}
We have $\Phi^\emptyset_{v^{(1,0)},\beta}(x)$ from (7.10), so it remains to compute $\Phi^{\{p\}}_{v^{(1,0)},\beta}(x)$ for $p=1,2,3,4$.  By (6.2) we have
\begin{align*}
\Phi^{\{1\}}_{v^{(1,0)},\beta}(x) &= x^{v^{(1,0)}}\sum_{z\in{\mathbb Z}} M_{z,1}(0)M_{z,0}(\sigma-1)M_{-z,0}(-\theta_1)M_{-z,0}(-\theta_2)x_0^z \\
\Phi^{\{2\}}_{v^{(1,0)},\beta}(x) &= x^{v^{(1,0)}}\sum_{z\in{\mathbb Z}} M_{z,0}(0)M_{z,1}(\sigma-1)M_{-z,0}(-\theta_1)M_{-z,0}(-\theta_2)x_0^z \\
\Phi^{\{3\}}_{v^{(1,0)},\beta}(x) &= x^{v^{(1,0)}}\sum_{z\in{\mathbb Z}} M_{z,0}(0)M_{z,0}(\sigma-1)M_{-z,1}(-\theta_1)M_{-z,0}(-\theta_2)x_0^z \\
\Phi^{\{4\}}_{v^{(1,0)},\beta}(x) &= x^{v^{(1,0)}}\sum_{z\in{\mathbb Z}} M_{z,0}(0)M_{z,0}(\sigma-1)M_{-z,0}(-\theta_1)M_{-z,1}(-\theta_2)x_0^z.
\end{align*}

We now apply (4.6)--(4.11).   We have $M_{z,0}(\sigma-1) = 0$ for $z\leq -\sigma$ by (4.9) and $M_{z,1}(0)=0$ for $z=0$ by (4.6), so
\begin{multline}
\Phi^{\{1\}}_{v^{(1,0)},\beta}(x) = - x_2^{\sigma-1}x_3^{-\theta_1}x_4^{-\theta_2}\\  
\cdot\bigg(\sum_{z=-\sigma+1}^{-1} \frac{(-z-1)!(1-\sigma)_{-z}}{(1-\theta_1)_{-z}(1-\theta_2)_{-z}}x_0^z 
+ \sum_{z=1}^\infty \frac{(\theta_1)_z(\theta_2)_z}{(\sigma)_zz!}\bigg(\sum_{s=0}^{z-1} \frac{1}{1+s}\bigg)x_0^z\bigg).
\end{multline}
We have $M_{z,0}(0) = 0$ for $z<0$ by (4.9) and $M_{z,1}(\sigma-1)=0$ for $z=0$ by (4.6), so
\begin{equation}
\Phi^{\{2\}}_{v^{(1,0)},\beta}(x) = -x_2^{\sigma-1}x_3^{-\theta_1}x_4^{-\theta_2}\sum_{z=1}^\infty 
\frac{(\theta_1)_z(\theta_2)_z}{(\sigma)_zz!}\bigg(\sum_{s=0}^{z-1} \frac{1}{\sigma+s}\bigg)x_0^z.
\end{equation}
Again, $M_{z,0}(0)=0$ for $z<0$ and $M_{-z,1}(-\theta_1) = 0$ for $z=0$ by (4.6), so
\begin{equation}
\Phi^{\{3\}}_{v^{(1,0)},\beta}(x) = -x_2^{\sigma-1}x_3^{-\theta_1}x_4^{-\theta_2}\sum_{z=1}^\infty 
\frac{(\theta_1)_z(\theta_2)_z}{(\sigma)_zz!}\bigg(\sum_{s=0}^{z-1}\frac{1}{\theta_1+s}\bigg)x_0^z.
\end{equation}
Similarly,
\begin{equation}
\Phi^{\{4\}}_{v^{(1,0)},\beta}(x) = -x_2^{\sigma-1}x_3^{-\theta_1}x_4^{-\theta_2}\sum_{z=1}^\infty 
\frac{(\theta_1)_z(\theta_2)_z}{(\sigma)_zz!}\bigg(\sum_{s=0}^{z-1}\frac{1}{\theta_2+s}\bigg)x_0^z.
\end{equation}

Substitution into (7.12) now gives the log solution when $\sigma\in{\mathbb Z}_{\geq 1}$:
\begin{multline}
x_2^{\sigma-1}x_3^{-\theta_1}x_4^{-\theta_2} \bigg(\sum_{z=0}^\infty \frac{(\theta_1)_z(\theta_2)_z}{(\sigma)_z z!} x^z_0\log x_0 \\ 
-\sum_{z=-\sigma+1}^{-1} \frac{(-z-1)!(1-\sigma)_{-z}}{(1-\theta_1)_{-z}(1-\theta_2)_{-z}}x_0^z  \\
+ \sum_{z=1}^\infty \frac{(\theta_1)_z(\theta_2)_z}{(\sigma)_zz!}\bigg(\sum_{s=0}^{z-1} \frac{1}{\theta_1+s} + \frac{1}{\theta_2+s} -\frac{1}{\sigma+s} -\frac{1}{1+s}\bigg)x_0^z\bigg).
\end{multline}
The initial monomial in (7.17) is the term corresponding to $z=-\sigma+1$ in the second summation.  This shows that $v^{(2,0)}$, which lies in $E_\beta$ but not in $E'_\beta$, is an exponent.

\section{Relation between $E'_\beta$ and $E'_\gamma$ when $\beta\equiv\gamma\pmod{{\mathbb Z}A}$}

Suppose that $\beta$ is nonresonant and $v\in E'_\beta$.  By Theorem~7.2 we get a full set of logarithmic series solutions for a parameter $\beta+u$ with $u\in{\mathbb Z}A$ from Equation~(6.3).  Fix $u\in{\mathbb Z}A$ and let $\gamma = \beta + u$.  Since $\gamma$ is also nonresonant, applying Theorem~7.2 to $v'\in E'_\gamma$ also gives a full set of logarithmic series solutions with parameter $\gamma = \beta+u$ that belong to the Nilsson ring.  The purpose of this section is to describe a relation between $E'_\beta$  and $E'_\gamma$ that can be used to make explicit the relation between the two sets of solutions with parameter $\gamma = \beta+u$, one that comes from $E'_\beta$ and one that comes from $E'_\gamma$.  Recall that for $v\in E'_\beta$ we put $M_v = \{i\in\{1,\dots,k\}\mid v_i\in{\mathbb Z}_{\geq 0}\}$.

\begin{proposition}
Suppose that $\beta$ is nonresonant and that $\gamma = \beta + u$ with $u\in{\mathbb Z}A$.  Let $v\in E'_\beta$.
There exists a unique $v'\in E'_\gamma$ such that $v'-v\in{\mathbb Z}^n$.  In particular, $M_v = M_{v'}$.  
\end{proposition}

\begin{proof}
Write $v=(v_1,\dots,v_n)$.  We have $v_i\in{\mathbb Z}_{\geq 0}$ for some $i\in\{1,\dots,k\}$.  To fix ideas, suppose that $v_1\in{\mathbb Z}_{\geq 0}$.  Write $u=\sum_{\sigma=1}^n \ell^{(u)}_\sigma{\bf a}_\sigma$ with $\ell^{(u)}_\sigma\in{\mathbb Z}$.  There is a unique $v'=(v'_1,\dots,v'_n)\in E_\gamma$ such that $v'_1\in\{0,1,\dots,\ell_1-1\}$ and $v_1'\equiv v_1+\ell^{(u)}_1\pmod{\ell_1}$.  By Lemma~3.1 there exists a unique $z_0\in{\mathbb Z}$ such that $\tilde{v}':=v' + z_0\ell$ lies in $E'_\gamma$.  Choose $z\in{\mathbb Z}$ such that $\tilde{v}'_1=v_1+\ell^{(u)}_1 +z\ell_1$.  We then have
\[ {\bf 0}=\gamma - \beta-u-z\ell = \sum_{i=2}^k (\tilde{v}'_i-v_i-\ell^{(u)}_i-z\ell_i){\bf a}_i + \sum_{j=k+1}^n (\tilde{v}_j' -v_j-\ell^{(u)}_j+z\ell_j){\bf a}_j. \]
But $A\setminus\{{\bf a}_1\}$ is a linearly independent set, so $\tilde{v}'_i-v_i-\ell^{(u)}_i-z\ell_i=0$ for $i=1,\dots,k$ and $\tilde{v}_j' -v_j-\ell^{(u)}_j+z\ell_j=0$ for $j=k+1,\dots,n$.  This shows that $\tilde{v}'-v=\ell^{(u)} +z\ell\in{\mathbb Z}^n$.
\end{proof}

Fix $v\in E'_\beta$, $u\in{\mathbb Z}A$, and $\gamma = \beta+u$.  By Proposition~8.1 there exists a  unique $v'\in E'_\gamma$ such that $v'-v =: \ell^{(u)}\in{\mathbb Z}^n$
with $u=\sum_{\sigma=1}^n \ell^{(u)}_\sigma{\bf a}_\sigma$.   
As an example, we give the connection between the two logarithm-free solutions $\Phi^\emptyset_{v,\beta+u}(\Lambda)$ and~$\Phi^\emptyset_{v',\gamma}(\Lambda)$ of the $A$-hypergeometric system with parameter $\gamma = \beta+u$.

The formula for $\Phi^\emptyset_{v,\beta+u}(\Lambda)$ is given by (7.3).  
To get the formula for $\Phi^\emptyset_{v',\gamma}(\Lambda)$, we apply (7.3) with $v$ replaced by $v'$, $\beta$ replaced by $\gamma$, and $u$ replaced by ${\bf 0}$, so that we can take the $\ell^{(u)}$ in that formula to be ${\bf 0}$.  We then have
\[ z_0 = \max\{ -v'_\mu/\ell_\mu\mid \mu\in M_v\}. \]
We have $-v'_\mu/\ell_\mu\leq 0$ for all $\mu\in M_v$, but since $v'\in E'_\gamma$  we have $v'_\mu<\ell_\mu$ for some $\mu\in M_v$, i.~e., $-v'_\mu/\ell_\mu>-1$ for that $\mu$, so $z_0\in(-1,0]$.  We thus get
\begin{equation}
\Phi^\emptyset_{v',\gamma}(x) = x^{v'}\sum_{z=0}^\infty \bigg(\prod_{i=1}^k M_{z\ell_i,0}(v'_i) \prod_{j=k+1}^n M_{-z\ell_j,0}(v'_j)\bigg)x_0^z.
\end{equation}
Multiplying by $\prod_{i=1}^k M_{\ell^{(u)}_i,0}(v_i)\prod_{j=k+1}^n M_{\ell^{(u)}_j,0}(v_j)$ this becomes
\begin{multline}
\bigg(\prod_{i=1}^k M_{\ell^{(u)}_i,0}(v_i)\prod_{j=k+1}^n M_{\ell^{(u)}_j,0}(v_j)\bigg)\Phi^\emptyset_{v',\gamma}(x) = \\ 
x^{v'}\sum_{z=0}^\infty \bigg(\prod_{i=1}^k M_{\ell^{(u)}_i,0}(v_i)M_{z\ell_i,0}(v'_i) \prod_{j=k+1}^n M_{\ell^{(u)}_j,0}(v_j)(M_{-z\ell_j,0}(v'_j)\bigg)x_0^z.
\end{multline}
Since $v_\sigma' = v_\sigma+\ell^{(u)}_\sigma$ for $\sigma=1,\dots,n$, the right-hand side can be simplified using the easily verified relation
\[ M_{a+b,0}(c) = M_{a,0}(c)M_{b,0}(a+c), \]
valid for all $a,b\in{\mathbb Z}$ provided $c$ and $a+c$ are not negative integers:
\begin{multline}
\bigg(\prod_{i=1}^k M_{\ell^{(u)}_i,0}(v_i)\prod_{j=k+1}^n M_{\ell^{(u)}_j,0}(v_j)\bigg)\Phi^\emptyset_{v',\gamma}(x) = \\ 
x^{v'}\sum_{z=0}^\infty \bigg(\prod_{i=1}^k M_{\ell^{(u)}_i+z\ell_i,0}(v_i) \prod_{j=k+1}^n M_{\ell^{(u)}_j-z\ell_j,0}(v_j)\bigg)x_0^z.
\end{multline}

To prove that (8.4) equals (7.3), we need to show that the summands on the right-hand side of (7.3) vanish when $z<0$.  For that, it suffices to show there exists $i\in\{1,\dots,k\}$ such that $M_{\ell^{(u)}_i+z\ell_i,0}(v_i) = 0$ for all $z<0$.  Since $v'\in E'_\gamma$, there is an $i\in\{1,\dots,k\}$ such that $v'_i-\ell_i\in{\mathbb Z}_{<0}$.  But $v'_i = v_i+\ell^{(u)}_i$, so $v_i+\ell^{(u)}_i+z\ell_i\in{\mathbb Z}_{<0}$ for all $z<0$.  Since $v_i\in{\mathbb Z}_{\geq 0}$, this implies that $M_{\ell^{(u)}_i+z\ell_i,0}(v_i) = 0$ by (4.9).  We conclude that
\begin{equation}
\Phi^\emptyset_{v,\beta+u}(x) = \bigg(\prod_{i=1}^k M_{\ell^{(u)}_i,0}(v_i)\prod_{j=k+1}^n M_{\ell^{(u)}_j,0}(v_j)\bigg)\Phi^\emptyset_{v',\gamma}(x),
\end{equation}
where $\ell^{(u)}=v'-v$.

\section{Maximal unipotent monodromy}

In this section we suppose that (1.3) holds, so $x_0=0$ is a regular singular point, and we investigate when there is ``maximal unipotent monodromy.''  By ``maximal unipotent monodromy'' in the classical case of an ODE with regular singularity at the origin, we mean that the monodromy about the origin is represented by a single unipotent Jordan block.  In the $A$-hypergeometric case, we interpret that to mean that there are log solutions involving powers $\log^\mu x_0$ for $\mu=0,1,\dots,{\rm vol}(A)-1$ ($=\sum_{i=1}^k \ell_i-1$, since we are assuming~(1.3)) and that the coefficients of the powers of $\log x_0$ are Nilsson series involving only integral powers of $x_0$.  

In the case of what Katz \cite{K} defines as ``generalized hypergeometric equations,'' one has maximal unipotent monodromy at the origin for an irreducible equation if and only if all exponents there are integers (see \cite[Chapter 3]{K}).  In the $A$-hypergeometric situation, maximal unipotent monodromy occurs exactly when the set $E'_\beta$ is a singleton (Theorem~9.1), one does not need any additional conditions on the nature of the unique element of $E'_\beta$.

In the classical case, some authors define maximal unipotent monodromy at the origin to mean that all exponents there equal 0 (see Almkvist and Zudilin \cite{AZ}).  This implies that the coefficient of the highest power of $\log$ is holomorphic and nonvanishing at the origin and that the coefficients of lower powers of $\log$ are holomorphic at the origin.  For this reason, we also determine when the Nilsson series coefficients of powers of $\log x_0$ contain only nonnegative integral powers of~$x_0$.  

We assume that the parameter $\beta$ is nonresonant, so by Theorem~7.2 the hypothesis of Corollary 6.6 is satisfied and there are $\sum_{i=1}^k \ell_i$ linearly independent solutions of the $A$-hypergeometric system with parameter $\beta+u$.  The largest power of $\log x_0$ that can appear in any of these solutions is 
$\sum_{i=1}^k \ell_i-1$ and that happens exactly when $E'_\beta$ is a singleton.  That establishes the ``only if'' direction of the following equivalence.
\begin{theorem}
Suppose that $(1.3)$ holds and that $\beta$ is nonresonant.  The $A$-hypergeometric system with parameter $\beta$ has maximal unipotent monodromy at the origin if and only if $E_\beta'$ is a singleton.
\end{theorem}

Before proving the ``if'' direction of Theorem 9.1, we examine when $E'_\beta$ is a singleton.
\begin{proposition}
The set $E'_\beta$ is a singleton if and only if the following two conditions are satisfied: \\
{\bf (a)} $\beta\in\sum_{i=1}^k ({\mathbb Z}_{\geq 0}){\bf a}_i + \sum_{j=k+1}^n{\mathbb C}{\bf a}_j$, \\
{\bf (b)} $\ell_i=1$ for $i=1,\dots,k$.
\end{proposition}

{\bf Remark.}  Note that condition (a) is equivalent to assuming
\[ \beta\in\sum_{i=1}^k {\mathbb Z}{\bf a}_i + \sum_{j=k+1}^n{\mathbb C}{\bf a}_j \]
because we can add multiples of $\sum_{i=1}^k \ell_i{\bf a}_i -\sum_{j=k+1}^n \ell_j{\bf a}_j$ to such an expression for $\beta$.  

\begin{proof}[Proof of Proposition $9.2$]
Suppose that $E'_\beta = \{v\}$.  It follows from Lemma~3.1 that $v=\tilde{v}^{(i,b)}$ for all $i\in\{1,\dots,k\}$ and $b\in\{0,1,\dots,\ell_i-1\}$.  We thus have $v_i\in{\mathbb Z}_{\geq 0}$ for $i=1,\dots,k$, which proves part (a).  If $b\neq b'$, then 
$\tilde{v}^{(i,b)}\neq \tilde{v}^{(i,b')}$ because their $i$-th coordinates are not congruent modulo $\ell_i$.  So if $E'_\beta$ is a singleton, then the sets $\{0,1,\dots,\ell_i-1\}$ must also be singletons, i.~e., the assertion of part (b) must hold.

Conversely, suppose that (a) and (b) hold.  By part (a) we have
\[ \beta = \sum_{i=1}^k v_i{\bf a}_i + \sum_{j=k+1}^n v_j{\bf a}_j \]
with $v_i\in{\mathbb Z}_{\geq 0}$ for $i=1,\dots,k$ and $v_j\in{\mathbb C}$ for $j=k+1,\dots,n$.
Put $v=(v_1,\dots,v_n)$. There exists a unique $z_0\in{\mathbb Z}$ such that $v':=v+z_0\ell\in E'_\beta$, namely, choose $z_0$ to be the smallest integer such that $v_i+z_0\ell_i\in{\mathbb Z}_{\geq 0}$ for $i=1\dots,k$.  Since $v_i+z_0\ell_i\in{\mathbb Z}_{\geq 0}$ for $i=1,\dots,k$ we have $m_{v'} = k$.  By part (b) we have $\sum_{i=1}^k \ell_i = k$, so Equation (1.4) implies that $v'$ is the only element of $E'_\beta$.
\end{proof}

\begin{proof}[Proof of Theorem $9.1$]
We now suppose that $E'_\beta$ is a singleton, say, $E'_\beta = \{v\}$, and that $\beta$ is nonresonant.  By Theorem~7.2, Equation~(6.3) gives us a full set of solutions in a Nilsson ring for the parameter $\beta+u$ involving powers $\log^\mu x_0$ for $\mu=0,1,\dots,k-1$.  We focus attention on the solutions for the parameter $\beta$ itself and check that the coefficients of powers of $\log x_0$ contain only integral powers of $x_0$. These coefficients are given by taking $u={\bf 0}$ and $\ell^{(u)}={\bf 0}$ in (6.2):
\begin{multline}
\Phi^Q_{v,\beta}(x) = \\ 
x^{v} \sum_{z\in{\mathbb Z}_{v,\bar{Q}^{\rm c}}({\bf 0},{\bf 0})} \bigg(\prod_{i=1}^k M_{\ell^{(u)}_i+z\ell_i,\rho_Q(i)}(v_i)\prod_{j=k+1}^nM_{\ell^{(u)}_j-z\ell_j,\rho_Q(j)}(v_j)\bigg)x_0^z.
\end{multline}
It follows from Proposition~9.2(b) that $v_i=0$ for some $i\in\{1,\dots,k\}$, so the factor $x^v$ contributes no powers of $x_0$.  Thus this series contains only integral powers of $x_0$.  
\end{proof}

We also determine when the coefficients $\Phi^Q_{v,\beta}(x)$ contain only nonnegative powers of $x_0$.  By Equation (9.3),
we need to determine when the sets ${\mathbb Z}_{v,\bar{Q}^{\rm c}}({\bf 0},{\bf 0})$ contain no negative integers
for all sequences $Q$ of length less than or equal to $k-1$.  Recall from (6.1) that 
\[ {\mathbb Z}_{v,\bar{Q}^{\rm c}}({\bf 0},{\bf 0}) = \{z\in{\mathbb Z}\mid \text{$\bar{Q}^{\rm c}$-nsupp$(v+z\ell) = \bar{Q}^{\rm c}$-nsupp$(v)$}\}. \]
\begin{lemma}
We have ${\mathbb Z}_{v,\bar{Q}^{\rm c}}({\bf 0},{\bf 0})\subseteq{\mathbb Z}_{\geq 0}$ for all sequences $Q$ of length less than or equal to $k-1$ if and only if $v_i=0$ for $i=1,\dots,k$.
\end{lemma}

\begin{proof}
Suppose that $v_i=0$ for $i=1,\dots,k$.  If $Q$ is any sequence of length less than or equal to $k-1$, then $i_0\in\bar{Q}^{\rm c}$ for some $i_0$, $1\leq i_0\leq k$.  Since $v_{i_0}=0$, we have $i_0\not\in\bar{Q}^{\rm c}$-nsupp$(v)$.  But for any $z\in{\mathbb Z}_{<0}$ we have $v_{i_0}+z\ell_{i_0} = z\in{\mathbb Z}_{<0}$, so $i_0\in\bar{Q}^{\rm c}$-nsupp$(v+z\ell)$.  Thus $z\not\in 
{\mathbb Z}_{v,\bar{Q}^{\rm c}}({\bf 0},{\bf 0})$.

Suppose that $v_i\neq 0$ for some $i\in\{1,\dots,k\}$.  To fix ideas, suppose that $v_i\neq 0$ for $i=1,\dots,k'$ and $v_i=0$ for $i=k'+1,\dots,k$.  Let $Q$ be the sequence $(k'+1,\dots,k)$, a sequence of length less than or equal to $k-1$.  We have $1,\dots,k'\in\bar{Q}^{\rm c}$.  Since $v\in E'_\beta$, we also know that $v_i$ is not a negative integer for $i=1,\dots,k$.  This implies that either $v_i$ is a positive integer or $v_i$ is not an integer for $i=1,\dots,k'$.  Furthermore, by Proposition~7.1, $v_j\not\in{\mathbb Z}$ for $j=k+1,\dots,n$.  It follows that 
\[ \bar{Q}^{\rm c}\text{-nsupp}(v-\ell) = \bar{Q}^{\rm c}\text{-nsupp}(v), \]
hence $-1\in {\mathbb Z}_{v,\bar{Q}^{\rm c}}({\bf 0},{\bf 0})$.
\end{proof}

From Theorem 9.1, Proposition 9.2, and Lemma 9.4 we get the following result.
\begin{corollary}
Suppose that $(1.3)$ holds and that $\beta$ is nonresonant.  The $A$-hypergeometric system with parameter $\beta$ has maximal unipotent monodromy and the coefficients of powers of $\log x_0$ contain only nonnegative powers of $x_0$ if and only if the following two conditions are satisfied: \\
{\bf (a)} $\beta\in\sum_{j=k+1}^n {\mathbb C}{\bf a}_j$, \\
{\bf (b)} $\ell_i=1$ for $i=1,\dots,k$.
\end{corollary}

Note that when conditions (a) and (b) are satisfied we have $E'_\beta=\{v\}$ with $v_i=0$ for $i=1,\dots,k$.  This implies that $E_\beta = \{v\}$ also.

Consider the classical one-variable hypergeometric equation for the series
\[ {}_{d+1-k}F_{k-1}(\alpha_1,\dots,\alpha_{d+1-k};\gamma_1,\dots,\gamma_{k-1};x) \]
described in Section~1, where the corresponding value of $\beta$ is
\[ \beta = (-\alpha_1,\dots,-\alpha_{d+1-k},\gamma_1-1,\dots,\gamma_{k-1}-1). \]
As noted in Section 1, $x_0=0$ is a regular singularity exactly when $k\geq d+1-k$.  We also noted there that $\ell_i = 1$ for $i=1,\dots,k$, so condition~(b) of Proposition~9.2 is satisfied.  Furthermore, using the set $A$ from that example, we have
\[ \beta = \sum_{i=2}^k (\gamma_{k-i+1}-1){\bf a}_i + \sum_{j=k+1}^{d+1} (-\alpha_{d+2-j}){\bf a}_j. \]
Condition (a) of Proposition 9.2 is satisfied if and only if $\gamma_i\in{\mathbb Z}_{\geq 0}$ for $i=1,\dots,k-1$, which is the condition for this differential equation to have maximal unipotent monodromy.  And by condition (a) of Corollary 9.5 the coefficients of powers of $\log x$ will be holomorphic at the origin if and only if $\gamma_i=1$ for $i=1,\dots,k-1$.  

One can also create examples by choosing $n-1$ independent vectors ${\bf a}_1,\dots,{\bf a}_{n-1}\in{\mathbb Z}^d$ and defining ${\bf a}_n=-\sum_{\mu=1}^{n-1} {\bf a}_\mu$.  Then $\sum_{\mu=1}^n{\bf a}_\mu = {\bf 0}$, so condition (b) of Proposition~9.2 is satisfied.  Furthermore, in this case the origin is an interior point of $\Delta(A)$ so all $\beta\in V_{\mathbb C}$ are nonresonant.  All $\beta\in{\mathbb Z}A$ satisfy condition (a) of Proposition~9.2.   Only $\beta = {\bf 0}$ satisfies condition (a) of Corollary 9.5. 

\begin{comment}
We make one further observation.  Consider the Ehrhart series 
\[ P(t):= \sum_{s=0}^\infty {\rm card}(s\Delta(A)\cap{\mathbb Z}A)t^s, \]
where $s\Delta(A)$ denotes the dilation of $\Delta(A)$ by the factor $s$.  
\begin{proposition}
Suppose that $(1.3)$ holds.  Then $\ell_i=1$ for $i=1,\dots,k$ (i.~e., condition (b) of Proposition $1.3$ holds) if and only if
\begin{equation}
P(t) = \frac{1 + t + \cdots + t^{k-1}}{(1-t)^n}.
\end{equation}
\end{proposition}

\begin{proof}
If (9.5) holds, then ${\rm vol}(A) = k$ so (1.3) implies $\ell_i=1$ for $i=1,\dots,k$.  Conversely, if $\ell_i=1$ for $i=1,\dots,k$, then Lemma~2.10 implies that ${\rm vol}(\Delta_i) = 1$ for $i=1,\dots,k$.  This implies in turn that the set $A\setminus\{{\bf a}_i\}$ is a basis for the lattice ${\mathbb Z}A$, so
\[ \sum_{s=0}^\infty {\rm card}(s\Delta_i\cap{\mathbb Z}A)t^s = \frac{1}{(1-t)^n}. \]
It follows that for each $I\subseteq\{1,\dots,k\}$ the set $A\setminus\{{\bf a}_i\}_{i\in I}$ is a basis for the abelian group ${\rm span}_{\mathbb R}(A\setminus\{{\bf a}_i\}_{i\in I})\cap{\mathbb Z}A$, so
\[ \sum_{s=0}^\infty {\rm card}(s\Delta_I\cap{\mathbb Z}A)t^s = \frac{1}{(1-t)^{n+1-|I|}}. \]
Equation (9.5) then follows from an inclusion-exclusion argument using Lemmas~2.2 and~2.5.
\end{proof}
\end{comment}

\end{document}